\documentclass[12pt]{article}

\usepackage{amssymb}
\usepackage{amsmath}
\usepackage{mathrsfs}
\usepackage{amsthm}
\usepackage{amsfonts}
\usepackage{color}
\usepackage{latexsym}
\usepackage{txfonts}
\usepackage{indentfirst}

\allowdisplaybreaks

\def\XXint#1#2#3{{\setbox0=\hbox{$#1{#2#3}{\int}$ }
\vcenter{\hbox{$#2#3$ }}\kern-.6\wd0}}

\newtheorem{theorem}{Theorem}[section]

\newtheorem{corollary}[theorem]{Corollary}
\newtheorem{proposition}[theorem]{Proposition}

\theoremstyle{definition}
\newtheorem{remark}[theorem]{Remark}
\newtheorem{definition}[theorem]{Definition}
\newtheorem{assumption}[theorem]{Assumption}
\renewcommand{\appendix}{\par
   \setcounter{section}{0}%
   \setcounter{subsection}{0}%
   \setcounter{subsubsection}{0}%
   \gdef\thesection{\@Alph\c@section}%
   \gdef\thesubsection{\@Alph\c@section.\@arabic\c@subsection}%
   \gdef\theHsection{\@Alph\c@section.}%
   \gdef\theHsubsection{\@Alph\c@section.\@arabic\c@subsection}%
   \csname appendixmore\endcsname
 }

\numberwithin{equation}{section}

\begin{document}
\title{\bf\Large Product Hardy spaces meet ball quasi-Banach function spaces
\footnotetext{\hspace{-0.35cm} 2010 {\it Mathematics Subject Classification}. Primary 42B30;
Secondary 42B25, 42B20.\endgraf
{\it Key words and phrases.} Product Hardy space, ball quasi-Banach function spaces, discrete Calder\'on's identity, Littlewood--Paley function, singular integral operators, extrapolation.\endgraf}}
\author{Jian Tan}
\date{}
\maketitle
	
\vspace{-0.8cm}

\begin{center}
\begin{minipage}{13cm}
{\small {\bf Abstract}\quad
The main purpose of this paper is to develop the theory of product Hardy spaces built on Banach lattices on $\mathbb R^n\times\mathbb R^m$.
First we introduce new product Hardy spaces ${H}_X(\mathbb R^n\times\mathbb R^m)$ associated with ball quasi-Banach function spaces $X(\mathbb R^n\times\mathbb R^m)$ via applying the Littlewood-Paley-Stein theory.
Then we establish a decomposition theorem for ${H}_X(\mathbb R^n\times\mathbb R^m)$ in terms of the discrete Calder\'on's identity. Moreover, we explore some useful and general extrapolation theorems of Rubio de Francia on $X(\mathbb R^n\times\mathbb R^m)$ and give some applications to boundedness of operators.  Finally, we conclude that the two-parameter singular integral operators $\widetilde T$ are bounded from ${H}_X(\mathbb R^n\times\mathbb R^m)$ to itself
and bounded from ${H}_X(\mathbb R^n\times\mathbb R^m)$ to $X(\mathbb R^n\times\mathbb R^m)$ via extrapolation. The main results obtained in this paper have a wide range of generality.
Especially, we can apply these results to many concrete examples of ball quasi-Banach function spaces, including product Herz spaces, weighted product Morrey spaces and product Musielak--Orlicz--type spaces.}
\end{minipage}
\end{center}

\arraycolsep=1pt

\vspace{0.2cm}

\section{Introduction }\label{intro} 

The pioneer work on the classical Hardy space $H^p$ was initiated by Stein and Weiss \cite{sw60} and systematically developed by Fefferman and Stein \cite{FS}. The classical Hardy spaces $H^p(\mathbb{R}^{n})$ is a suitable substitute of Lebesgue spaces $L^p(\mathbb{R}^{n})$ with $p\in(0,1]$ when studying the boundedness of operators, which is built on the Lebesgue space.
Meanwhile,  due to the need from many applications,
various variants of classical Hardy spaces have been investigated extensively, including weighted Hardy spaces, Herz-Hardy spaces, Hardy-Morrey spaces, Hardy-Orlicz spaces, Musielak-Orlicz Hardy spaces and variable Hardy spaces. 
The known Orlicz spaces, Lorentz spaces and variable Lebesgue spaces are quasi-Banach function spaces. However, Morrey spaces and weighted Lebesgue spaces are not necessarily quasi-Banach function spaces. In order to include all these function spaces in a unified framework, Sawano et al. \cite{SHYY17} introduced the ball quasi-Banach function space, which is defined similarly to quasi-Banach function spaces except that Lebesgue measurable sets are replaced by balls in $\mathbb{R}^{n}$. 
Compared with (quasi-)Banach function spaces, ball (quasi-)Banach function spaces contain more function spaces and hence are more general.
Then Sawano et al. \cite{SHYY17} introduced the Hardy type space $H_{X}\left(\mathbb{R}^{n}\right)$. For more research works on Hardy spaces associated with ball quasi-Banach function spaces, see for instance \cite{CWYZ,HuangChangYang, YHYY1, YHYY2}. 

On another hand, product Hardy spaces were first developed by Gundy and Stein \cite{GS}. 
Chang and Fefferman \cite{CF1,CF2} gave many important results on $H^p$ theory on the polydisc, BMO on product domains and the atomic decomposition of product Hardy spaces. 
Ding et.al \cite{DHLW} establish the
$(H_{w}^{p}(\mathbb R^n\times \mathbb R^m), L_{w}^{p}(\mathbb R^n\times \mathbb R^m))-$boundedness for two-parameter singular integrals by using the discrete Calder\'on's identity and weighted Littlewood-Paley-Stein theory.
Recently, there are some recent development on product Hardy spaces by Han et al. in \cite{HLPW,HLW16}.
Although the real-variable theory of Hardy spaces associated with ball quasi-Banach function spaces in one-parameter has been very fruitful,
the corresponding real-variable theory of product Hardy type spaces is still absent.
A natural question arises: can one develop the theory of the product Hardy spaces built on Banach lattices on $\mathbb R^n\times\mathbb R^m$
fulfilling some minor assumptions?  

In this paper, we answered in the affirmative by introducing the product Hardy type spaces $H_{X}\left(\mathbb{R}^{n}\times \mathbb R^m\right)$. 
To achieve our goal, we devote to establishing the real-variable theory of product Hardy spaces built on Banach lattices on $\mathbb R^n\times\mathbb R^m$
fulfilling some minor assumptions.
The key idea used in our paper is to employ the Littlewood--Paley--Stein theory
to define the appropriate product square functions and product Hardy spaces.
Also, since that the quasi-norm associated with $X(\mathbb R^n\times\mathbb R^m)$ has no explicit expression, we make full use of the minor assumptions of vector-valued maximal inequality $X(\mathbb R^n\times\mathbb R^m)$
and the boundedness the strong Hardy--Littlewood maximal operator $M_s$ on the associate space of the convexification of $X(\mathbb R^n\times\mathbb R^m)$. 
Moving into another direction, the atomic decomposition of distributions in $H^p(\mathbb R)$ established by Coifman \cite{Co}
was extended by Latter \cite{La} to the spaces $H^p(\mathbb R^n)$.
With the help of the discrete Calder\'on's identity
and the similar techniques involving variable Hardy spaces which we have used in \cite{Tan20, Tan1, T22}, we also obtain a decomposition theorem for ${H}_X(\mathbb R^n\times\mathbb R^m)$. 
Moreover, when we consider the boundedness of operators, the refined extrapolation plays an key role in our main result.
Finally, we remark that the main theorems obtained in this paper 
can be applied to many concrete product Hardy type spaces, including product Herz spaces, product Morrey spaces and product Musielak--Orlicz--type spaces. To the best of our knowledge, 
even in the special cases of product Herz spaces, weighted product Morrey spaces and product Musielak--Orlicz--type spaces,
the real-variable theory of these product Hardy type spaces and their applications to the boundedness of the two-parameter singular integral operators obtained in our paper are completely new.
It is also worth pointing out that the new approach for defining the product Hardy
type spaces in a very general framework also works for the local product Hardy type spaces, the product Hardy spaces associated with operators, flag Hardy spaces and the anisotropic product Hardy spaces
associated with ball quasi-Banach function spaces. The details will be given in our coming papers.

The organization of this article is as follows.
In Section 2, we first recall the definitions and known results of
product test functions and product singular integrals.
Then we introduce the ball quasi-Banach function spaces on product domain.
Some necessary assumptions for strong maximal operators on ball quasi-Banach function spaces were also restated in this section.
Section 3 concerns product Hardy spaces associated with ball quasi-Banach function spaces defined by the continuous and discrete Littlewood–Paley functions. Moreover, we obtain the equivalent Littlewood–Paley characterizations by applying the Plancherel–P\'olya type inequalities.
Section 4 is devoted to the decomposition for the product Hardy spaces associated with ball quasi-Banach function spaces in terms of the discrete Calder\'on identity.
In Section 5, we establish some useful and general extrapolation theorems
on product ball (quasi-)Banach function spaces.
As one of the applications, we explore a general approach to derive $H_X\rightarrow X$ boundedness
from $H_X\rightarrow H_X$ boundedness of linear operators. 
In Section 6, we establish the ${H}_X(\mathbb R^n\times\mathbb R^m)$ boundedness of the two-parameter singular integral operator $\widetilde T$ via extrapolation. Then we obtain that the two-parameter singular integral operator $\widetilde T$ are bounded from ${H}_X(\mathbb R^n\times\mathbb R^m)$ to itself and from ${H}_X(\mathbb R^n\times\mathbb R^m)$ to $X(\mathbb R^n\times\mathbb R^m)$.

\section{Preliminaries}
We first make some conventions on notation. 
The symbol $C$ denotes a positive constant which is independent of the main parameters, but it may vary from
line to line.
For any set $E$ of $X$, we use $\chi_E$ to denote its characteristic function and $E^c$ the set
$X\setminus E$. 
In this section, we recall some notation and necessary results on the theory of product singular integrals and ball Banach function spaces on product domain. 

\subsection{Product test functions and product singular integrals}
Now let us recall the class of ``bi-parameter test functions''
which was introduced in \cite{DHLW}.

\begin{definition}\label{s1de1}\quad
For a positive integer $M$,
a function $f$ defined on $\mathbb R^n\times\mathbb R^m$ is said to be a test function in $\mathcal S_M(\mathbb R^n\times\mathbb R^m)$ if $f$ satisfies\\
(i) For multi-indices $\alpha,\,\beta$ with $|\alpha|,\,|\beta|\le M-1$,
\begin{align*}
|D_x^\alpha D_y^\beta f(x,y)|\leq C \frac{1}{(1+|x|)^{n+M+|\alpha|}}\frac{1}{(1+|y|)^{n+M+|\beta|}};
\end{align*}
(ii) For $|x-x'|\leq \frac{1}{2}(1+|x|)$, $|v|=M$ and $|\beta|\le M-1$,
\begin{align*}
|D_x^vD^\beta_yf(x,y)-D_x^vD_y^\beta f(x',y)|\leq C \frac{|x-x'|}{(1+|x|)^{n+2M}}\frac{1}{(1+|y|)^{n+M+|\beta|}};
\end{align*}
(iii) For $|y-y'|\leq \frac{1}{2}(1+|y|)$, $|v|=M$ and $|\alpha|\le M-1$,
\begin{align*}
|D_x^\alpha D^v_yf(x,y)-D_x^\alpha D_y^v f(x,y')|\leq C \frac{1}{(1+|x|)^{n+M+|\alpha|}}\frac{|y-y'|}{(1+|y|)^{n+2M}};
\end{align*}
(iv) For $|x-x'|\leq \frac{1}{2}(1+|x|)$, $|y-y'|\leq \frac{1}{2}(1+|y|)$ and $|v|=M$
\begin{align*}
|D_x^vD^v_yf(x,y)-D_x^vD_y^v f(x',y)-D_x^vD^v_yf(x,y')+D_x^vD_y^v f(x',y')|\leq C \frac{|x-x'|}{(1+|x|)^{n+2M}}\frac{|y-y'|}{(1+|y|)^{n+2M}};
\end{align*}
(v) For multi-indices $\alpha,\,\beta$ with $|\alpha|,\,\beta\le M-1$,
$$
\int_{\mathbb R^n}f(x,y)x^\alpha dx=\int_{\mathbb R^m}f(x,y)y^\beta dy=0.
$$
If $f$ is a test function in $\mathcal S_M(\mathbb R^n\times\mathbb R^m)$ and the norm of $f$ in $\mathcal S_M$ is defined by
$$
\|f \| _{\mathcal S_M(\mathbb R^n\times \mathbb R^m)}=
\inf \{ C> 0: (\mbox{i})~
\mbox{--}~(\mbox{iv})~\mbox{hold}~ \}.
$$
\end{definition}
Observe that $\mathcal S_M(\mathbb R^n\times\mathbb R^m)$ is a Banach space.
We denote by $(\mathcal S_M)^\prime(\mathbb R^n\times\mathbb R^m)$ the dual space of $\mathcal S_M(\mathbb R^n)$.
We now consider the following two-parameter singular integral operators $\widetilde{T}$ on $\mathbb{R}^n \times \mathbb{R}^m$, defined by
$$
\widetilde{T}(f)\left(x_1, x_2\right)=\text { p.v. } \int_{\mathbb{R}^m \times \mathbb{R}^n} K\left(x_1-y_1, x_2-y_2\right) f\left(y_1, y_2\right) d y_1 d y_2 .
$$
where the kernel $K$ is a distribution on $\mathbb{R}^n \times \mathbb{R}^m$, which coincides with a $C^{\infty}$ function away from the coordinate subspaces $x_j=0\,(j=1,2)$ and fulfills that
\\
(1) (Differential inequalities) For each multi index $\alpha=\left(\alpha_1, \alpha_2\right)$, there is a constant $C_\alpha$ such that
$$
\left|\partial_{x_1}^{\alpha_1} \partial_{x_2}^{\alpha_2} K\left(x_1, x_2\right)\right| \leq C_\alpha\left|x_1\right|^{-n-\left|\alpha_1\right|}\left|x_2\right|^{-m-\left|\alpha_2\right|}
$$
(2) (Cancellation condition) For any normalized bump function $\varphi_1$ on $\mathbb{R}^n$ and any $R>0$, the distribution
$$
\int_{\mathbb{R}^n} K\left(x_1, x_2\right) \varphi_1\left(R x_1\right) d x_1
$$
is a one-parameter kernel on $\mathbb{R}^m$ and similarly for any normalized bump function $\varphi_2$ on $\mathbb{R}^m$ and any $R>0$, the distribution
$$
\int_{\mathbb{R}^m} K\left(x_1, x_2\right) \varphi_2\left(R x_2\right) d x_2
$$
is a one-parameter kernel on $\mathbb{R}^n$. Furthermore, for any normalized bump function $\psi$ on $\mathbb{R}^n \times \mathbb{R}^m$ and any $R>0$, we have
$$
\left|\int_{\mathbb{R}^n \times \mathbb{R}^m} K\left(x_1, x_2\right) \psi\left(R x_1, R x_2\right) d x_1 d x_2\right| \leq C,
$$
where $C$ is a constant independent of $\psi$ and $R$.
In \cite{FS1}, Fefferman and Stein established the $L_w^p$ boundedness for bi-parameter singular integrals. Then the $\left(H_w^p\left(\mathbb{R}^n \times \mathbb{R}^m\right), H_w^p\left(\mathbb{R}^n \times \mathbb{R}^m\right)\right)$ boundedness and $\left(H_w^p\left(\mathbb{R}^n \times \mathbb{R}^m\right), L_w^p\left(\mathbb{R}^{n+m}\right)\right)$ boundedness of the bi-parameter singular integrals were obtained in \cite{DHLW} under the considerably weaker assumption that $w$ is only in some product $A_q\left(\mathbb{R}^n \times \mathbb{R}^m\right)$ for some $1<q<\infty$.

\subsection{Ball quasi-Banach function spaces on product domain}

Denote by $\mathcal{M}$ the set of all measurable functions on $\mathbb{R}^{n}\times\mathbb R^m$. Now we introduce the definition of product Banach function spaces on $\mathbb{R}^{n}\times\mathbb R^m$.

\begin{definition}\label{de:01}
	\textup{A Banach space $Y(\mathbb{R}^{n}\times\mathbb R^m)\subset \mathcal{M}$ is called a Banach function space if it satisfies}
	\begin{itemize}
		\item[$\left( {\rm \romannumeral1} \right)$]
		\textup{$\left\|f\right\|_{Y(\mathbb{R}^{n}\times\mathbb R^m)}=0$ if and only if $f=0$ almost everywhere;}
		\item[$\left( {\rm \romannumeral2} \right)$]
		\textup{$\left|g\right| \le \left|f\right|$ almost everywhere implies that $\left\|g\right\|_{Y(\mathbb{R}^{n}\times\mathbb R^m)} \le \left\|f\right\|_{Y(\mathbb{R}^{n}\times\mathbb R^m)}$;}
		\item[$\left( {\rm \romannumeral3} \right)$]
		\textup{$0 \le f_{m}\uparrow f$ almost everywhere implies that $\left\|f_{m}\right\|_{Y(\mathbb{R}^{n}\times\mathbb R^m)} \uparrow \left\|f\right\|_{Y(\mathbb{R}^{n}\times\mathbb R^m)}$;}
		\item[$\left( {\rm \romannumeral4} \right)$]
		\textup{$\chi_{E} \in Y(\mathbb{R}^{n}\times\mathbb R^m)$ for any measurable set $E \subset \mathbb{R}^{n}$ with finite measure;}
		\item[$\left( {\rm \romannumeral5} \right)$]
		\textup{for any measurable set $E \subset \mathbb{R}^n\times \mathbb R^m$ with finite measure, there exists a positive constant $C_{\left(E\right)}$, depending on $E$, such that, for all $f\in Y(\mathbb{R}^{n}\times\mathbb R^m)$,}
		\begin{align}\label{eq:second1}
			\int_{E} \left|f\left(x,y\right)\right|\,dxdy \le C_{\left(E\right)} \left\|f\right\|_{Y(\mathbb{R}^{n}\times\mathbb R^m)},
		\end{align}
	\end{itemize}
	\textup{where $f$ and $g$ are measurable functions. In fact, the above conditions ensure that the norm $\left\| \cdot \right\|_{Y(\mathbb{R}^{n}\times\mathbb R^m)}$ is a product Banach function norm on $\mathbb{R}^{n}\times\mathbb R^m$.}
\end{definition}


 
In order to extend quasi-Banach function spaces further so that Morrey spaces and some other related spaces are included in this generalized framework,
we introduce the product ball quasi-Banach function spaces.
For $z=(x,y)\in \mathbb{R}^{n}\times\mathbb R^m$ and $r\in \left(0,\infty\right)$, let $R\left(z,r\right) := \left\{u\in \mathbb{R}^{n}\times\mathbb R^m: \left|z-u\right|<r \right\}$, and
\begin{align}\label{eq:second2}
	\mathbb{B}:=\left\{R\left(z,r\right):z\in\mathbb{R}^{n}\times\mathbb R^m\ {\rm and}\ r\in\left(0,\infty\right)\right\}.
\end{align}
We now present the notion of product ball quasi-Banach function spaces as follows, which is from \cite{SHYY17}.

\begin{definition}\label{de:03}
	\textup{We call $X(\mathbb{R}^{n}\times\mathbb R^m)\subset \mathcal{M}$ a product ball quasi-Banach function space if it satisfies}
	\begin{itemize}
		\item[$\left( {\rm \romannumeral1} \right)$]
		\textup{$\left\|f\right\|_{X(\mathbb{R}^{n}\times\mathbb R^m)}=0$ implies that $f=0$ almost everywhere;}
		\item[$\left( {\rm \romannumeral2} \right)$]
		\textup{$\left|g\right| \le \left|f\right|$ almost everywhere implies that $\left\|g\right\|_{X(\mathbb{R}^{n}\times\mathbb R^m)} \le \left\|f\right\|_{X(\mathbb{R}^{n}\times\mathbb R^m)}$;}
		\item[$\left( {\rm \romannumeral3} \right)$]
		\textup{$0 \le f_{m}\uparrow f$ almost everywhere implies that $\left\|f_{m}\right\|_{X(\mathbb{R}^{n}\times\mathbb R^m)} \uparrow \left\|f\right\|_{X(\mathbb{R}^{n}\times\mathbb R^m)}$;}
		\item[$\left( {\rm \romannumeral4} \right)$]
		\textup{$R\in \mathbb{B}$ implies that $\chi_{R}\in X(\mathbb{R}^{n}\times\mathbb R^m)$, where $\mathbb{B}$ is as in (\ref{eq:second2}).}
	\end{itemize}
\end{definition}

A product ball quasi-Banach function space $X(\mathbb{R}^{n}\times\mathbb R^m)$ is called a product ball Banach function space if the norm of $X(\mathbb{R}^{n}\times\mathbb R^m)$ satisfies the triangle inequality: for all $f$, $g\in X(\mathbb{R}^{n}\times\mathbb R^m)$,
\begin{align}\label{eq:second3}
	\left\|f+g\right\|_{X(\mathbb{R}^{n}\times\mathbb R^m)} \le \left\|f\right\|_{X(\mathbb{R}^{n}\times\mathbb R^m)} + \left\|g\right\|_{X(\mathbb{R}^{n}\times\mathbb R^m)},
\end{align}
and, for any $R\in\mathbb{B}$, there exists a positive constant $C_{\left(B\right)}$, depending on $R$, such that, for all $f\in X(\mathbb{R}^{n}\times\mathbb R^m)$,
\begin{align}\label{eq:second4}
	\int_{R} \left|f\left(x, y\right)\right|\,dxdy \le C_{\left(R\right)} \left\|f\right\|_{X(\mathbb{R}^{n}\times\mathbb R^m)}. 
\end{align}

\begin{remark}
	\textup{Here we give some specific examples of product ball quasi-Banach function spaces as follows. In addition to $L^{p}\left(\mathbb{R}^{n}\times \mathbb R^m\right)$, there are several examples of product ball quasi-Banach function spaces.}
	\begin{itemize}
		\item[$\left( {\rm a} \right)$]
		\textup{For any $q\in\left[1,\infty\right]$, denote by $A_{q, R}\left(\mathbb{R}^{n}\times\mathbb R^m\right)$ the class of all product Muckenhoupt weights (See Section 5 for a precise definition). For any $p\in\left(0,\infty\right)$ and $w\in A_{\infty, R}\left(\mathbb{R}^{n}\times \mathbb R^m\right)$, the weighted Lebesgue space $L_{w}^{p}\left(\mathbb{R}^{n}\times\mathbb R^m\right)$ is defined by setting}
		\begin{align*}
			L_{w}^{p}\left(\mathbb{R}^{n}\times\mathbb R^m\right):=\left\{f\ {\rm is\;  measurable}:\left\|f\right\|_{L_{w}^{p}\left(\mathbb{R}^{n}\times\mathbb R^m\right)}:=\left[\int_{\mathbb{R}^{n}\times\mathbb R^m} \left|f\left(x, y\right)\right|^{p}w\left(x, y\right)\,dxdy \right]^{1/p} <\infty \right\}.
		\end{align*}	
From the definition of $A_{\infty, R}(\mathbb R^n\times\mathbb R^m)$, we find that, for any 
$R\in\mathbb B$ with $\mathbb B$ as in (2.2), $\chi_R\in L^p_w(\mathbb R^n\times\mathbb R^m)$.
Therefore, the space $L^p_w(\mathbb R^n\times\mathbb R^m)$ is a product ball quasi-Banach function space.		
			
		\item[$\left( {\rm b} \right)$]
		\textup{For the cube $Q\left({\vec 0}_{n},1\right)$ and any $i\in\mathbb{N}$, let}
		\begin{align*}
			I_{i}\left(Q\left({\vec 0}_{n},1\right)\right):=Q\left({\vec 0}_{n},2^{i+1}\right)\setminus Q\left({\vec 0}_{n},2^{i}\right).
		\end{align*}
\textup{Similarly, for the cube $Q\left({\vec 0}_{m},1\right)$ and any $j\in\mathbb{N}$, let}
		\begin{align*}
			J_{j}\left(Q\left({\vec 0}_{m},1\right)\right):=Q\left({\vec 0}_{m},2^{j+1}\right)\setminus Q\left({\vec 0}_{n},2^{i}\right).
		\end{align*}
		\textup{Here, ${\vec 0}_{n}$ and ${\vec 0}_{m}$ denote the origins of $\mathbb{R}^{n}$ and $\mathbb{R}^{m}$, respectively.	
The product Herz space  $\vec{K}_q^{\alpha, p}\left(\mathbb{R}^n \times \mathbb{R}^m\right)$ is defined in \cite[Definition 2.1]{W}.
Let $\alpha \in \mathbb{R}$ and $0<p, q \leq \infty$. The product Herz space $\vec{K}_q^{\alpha, p}\left(\mathbb{R}^n \times \mathbb{R}^m\right)$ consists of all
$$
f \in L_{\text {loc }}^q\left(\mathbb{R}^n \times \mathbb{R}^m \backslash\{(0,0)\}\right)
$$
such that $\|f\|_{\vec{K}_q^{\alpha, p}\left(\mathbb{R}^n \times \mathbb{R}^m\right)}<\infty$, where
$$
\|f\|_{\vec{K}_q^{\alpha, p}\left(\mathbb{R}^n \times \mathbb{R}^m\right)}=\left\{\sum_{i, j \in \mathbb{Z}} 2^{(i+j) p \alpha}\left\|f \chi_{I_i \times J_j}\right\|_{L^q\left(\mathbb{R}^n \times \mathbb{R}^m\right)}^p\right\}^{\frac{1}{p}} .
$$
From the definition of $\vec{K}_q^{\alpha, p}\left(\mathbb{R}^n \times \mathbb{R}^m\right)$, it follows that, for any 
$R\in\mathbb B$ with $\mathbb B$ as in (2.2), $\chi_R\in \vec{K}_q^{\alpha, p}\left(\mathbb{R}^n \times \mathbb{R}^m\right)$.
Thus, the space $\vec{K}_q^{\alpha, p}\left(\mathbb{R}^n \times \mathbb{R}^m\right)$ is a product ball quasi-Banach function space.
When $p, q\in(1, \infty)$, 
$$(\vec{K}_q^{\alpha, p}\left(\mathbb{R}^n \times \mathbb{R}^m\right))^\ast=
\vec{K}_{q'}^{-\alpha, p'}\left(\mathbb{R}^n \times \mathbb{R}^m\right)$$
with equivalent norms.
We conclude that, for any $f\in \vec{K}_q^{\alpha, p}\left(\mathbb{R}^n \times \mathbb{R}^m\right)$,
$$
\left|\int_{\mathbb R^n\times\mathbb R^m}\chi_R(x,y)f(x,y)dxdy\right|\le \|\chi_R\|_{\vec{K}_{q'}^{-\alpha, p'}\left(\mathbb{R}^n \times \mathbb{R}^m\right)}\|f\|_{\vec{K}_q^{\alpha, p}\left(\mathbb{R}^n \times \mathbb{R}^m\right)}.
$$
From \cite[Proposition 2.1]{W}, the space $\vec{K}_q^{\alpha, p}\left(\mathbb{R}^n \times \mathbb{R}^m\right)$ is a quasi-Banach space, and if $p, q \geq 1$, then $\vec{K}_q^{\alpha, p}\left(\mathbb{R}^n \times \mathbb{R}^m\right)$ is a Banach space.
Therefore, when $p, q\in (1,\infty)$, the space $\vec{K}_q^{\alpha, p}\left(\mathbb{R}^n \times \mathbb{R}^m\right)$
is a product ball Banach function space.}  
		\item[$\left( {\rm c} \right)$]
		\textup{Let $0<p\le \infty$, $u:\mathbb{R}^{n}\times\mathbb R^m\rightarrow (0,\infty)$, and $w:\mathbb R^n\times\mathbb R^m\rightarrow (0,\infty)$ be a weight function. Then the weighted product Morrey space $\mathcal{M}_{u,p}^{w}\left(\mathbb{R}^{n}\times\mathbb R^m\right)$ is defined to be the set of all $f\in L_{{\rm loc}}^{q}\left(\mathbb{R}^{n}\times \mathbb R^m\right)$ such that}
		\begin{align*}
			\left\|f\right\|_{\mathcal{M}_{u,p}^{w}\left(\mathbb{R}^{n}\times\mathbb R^m\right)}:=\sup\limits_{R\subset\mathbb{R}^{n}\times\mathbb R^m} \frac{1}{u(R)}
\left[\int_{R} \left|f\left(x,y\right)\right|^{p}w(x,y)dxdy \right]^{\frac{1}{p}} <\infty,
		\end{align*}
		\textup{where the supremum is taken over all $R\subset \mathbb{R}^{n}\times \mathbb R^m$.}
When $u=1$, the space $\mathcal{M}_{u,p}^{w}\left(\mathbb{R}^{n}\times\mathbb R^m\right)$ reduces to the
classical weighted product Lebesgue space, while $\mathcal{M}_{u,p}^{w}\left(\mathbb{R}^{n}\times\mathbb R^m\right)$ is a product Morrey space for $w=1, u=|R|^{\frac{1}{p}-\frac{1}{q}}$ for $1<p\le q<\infty.$		
When $u\in \mathbb W_w^p$ (For precise definitions, see \cite[Theorem 3.1]{Ho17} or \cite[Definition 2.4]{W1}), then for any 
$R\in\mathbb B$ with $\mathbb B$ as in (2.2) we know that $\chi_B\in \mathcal{M}_{u,p}^{w}\left(\mathbb{R}^{n}\times\mathbb R^m\right)$. 
Thus, the space $\mathcal{M}_{u,p}^{w}\left(\mathbb{R}^{n}\times\mathbb R^m\right)$ is a product ball quasi-Banach function space.
Let $1<p<\infty$ and $u$ satisfies necessary conditions, the weighted product Morrey space and its
pre-dual spaces are Banach spaces. Similarly, we find that in this case the space $\mathcal{M}_{u,p}^{w}\left(\mathbb{R}^{n}\times\mathbb R^m\right)$ is also a product ball Banach function space.
For more details on weight product Morrey space, see \cite{Ho17,W1}.
     \item[$\left( {\rm d} \right)$]
     For a product growth function $\varphi$, similar to the one-parameter case
     in \cite{HH,YLK}, the product Musielak--Orlicz--type space $L^\varphi(\mathbb R^n\times \mathbb R^m)$ is defined to be the set of all measurable functions $f$ such that, for some $\lambda\in(0,\infty)$, $$
\int_{\mathbb R^n\times\mathbb R^m} \varphi(z,|f(z)|/\lambda)dz<\infty
     $$
     with the Luxembourg--Nakano (quasi-)norm
     $$
\|f\|_{L^\varphi(\mathbb R^m\times\mathbb R^m)}:=
\inf\left\{\lambda>0:\;\int_{\mathbb R^n\times\mathbb R^m} \varphi(z,|f(z)|/\lambda)dz\le 1\right\}.     
     $$
  For more precise definitions we refer the readers to \cite[Section 2]{FHLY}.
  When $\varphi$ satisfies the mild assumptions, from \cite[Lemma 3.6]{FHLY}, we see that 
 for any 
$R\in\mathbb B$ with $\mathbb B$ as in (2.2), $\chi_R\in L^\varphi(\mathbb R^m\times\mathbb R^m)$. 
Thus, the space $L^\varphi(\mathbb R^m\times\mathbb R^m)$ is a product ball quasi-Banach function space.
	\end{itemize}
\textup{}	
\end{remark}

\subsection{The strong Hardy-Littlewood maximal operator}

Denote by $L_{{\rm loc}}^{1}\left(\mathbb{R}^{n}\times \mathbb R^m \right)$ the set of all locally integrable functions on $\mathbb{R}^{n}\times \mathbb R^m$. 
In the product domain, it is natural to replace the classical Hardy-Littlewood maximal operator $M$ by the strong maximal function $M_s$, which is defined by setting, for all $f\in L_{{\rm loc}}^{1}\left(\mathbb{R}^{n}\times \mathbb R^m\right)$ and $x\in \mathbb{R}^{n}\times \mathbb R^m$,
\begin{align}\label{eq:second7}
	\mathcal M_sf\left(x\right):= \sup\limits_{x\in R} \frac{1}{|R|} \int_{R} \left|f\left(y\right)\right|\,dy,
\end{align}
where the supremum of $R$ are taken over all dyadic rectangles of the form $R=I\times J$ and where $I,\,J$ are cubes in $\mathbb R^n$ and $\mathbb R^m$,
respectively.

For any $\theta\in \left(0,\infty\right)$, the powered strong Hardy-Littlewood maximal operator $\mathcal M_s^{\left(\theta\right)}$ is defined by setting, for all $f\in L_{{\rm loc}}^{1}\left(\mathbb{R}^{n}\times \mathbb R^m\right)$ and $x\in \mathbb{R}^{n}\times \mathbb R^m$,
\begin{align}\label{eq:second8}
	\mathcal M_s^{\left(\theta\right)}\left(f\right)\left(x\right) := \left\{\mathcal M_s\left(\left|f\right|^{\theta}\right)\left(x\right) \right\}^{1/\theta}.
\end{align}

In order to prove several theorems in this paper, we need the following two assumptions and several lemmas. 

\begin{assumption}\label{as:01}
	\textup{Let $X$ be ball quasi-Banach function space.
		For some $\theta,s\in\left(0,1\right]$ and $\theta<s$, there exists a positive constant $C$ such that, for any  $\left\{f_{j}\right\}_{j=1}^{\infty} \subset L_{{\rm loc}}^{1}\left(\mathbb{R}^{n}\times \mathbb R^m\right)$,}
	\begin{align}\label{e2.6}
		\left\|\left\{\sum_{j=1}^{\infty}\left[\mathcal M_s^{(\theta)}\left(f_{j}\right)\right]^{s}\right\}^{1 / s}\right\|_{X} \le C\left\|\left\{\sum_{j=1}^{\infty}\left|f_{j}\right|^{s}\right\}^{1 / s}\right\|_{X}
	\end{align}
and
	\begin{align}\label{e2.7}
		\left\|\left\{\sum_{j=1}^{\infty}\left[\mathcal M_s^{(\theta)}\left(f_{j}\right)\right]^{s}\right\}^{1 / s}\right\|_{X^{\frac{s}{2}}} \le C\left\|\left\{\sum_{j=1}^{\infty}\left|f_{j}\right|^{s}\right\}^{1 / s}\right\|_{X^{\frac{s}{2}}}.
	\end{align}
\end{assumption}

\begin{remark}\label{re:07}
First, when $X:=L^{p}\left(\mathbb{R}^{n}\times\mathbb R^m\right)$, $p\in \left(1,\infty\right)$, $\theta=1$ and $s\in\left(1,\infty\right]$, $(\ref{e2.6})$ is called the well-known
Fefferman--Stein vector-valued strong maximal inequality for. 
Similar to the one-parameter case in \cite[Theorem 1]{FS71}, we also know that (\ref{e2.6}) also holds true when $\theta$, $s\in \left(0,1\right]$, $\theta<s$, $X:=L^{p}\left(\mathbb{R}^{n}\times\mathbb R^m\right)$ and $p\in\left(\theta,\infty\right)$. Since $\big(L^p(\mathbb R^n\times\mathbb R^m)\big)^{s/2}= L^{ps/2}(\mathbb R^n\times\mathbb R^m)$ for any $p, s\in(0,\infty)$, we conclude that  (\ref{e2.7}) also holds true when $\theta,s\in(0,1]$, $\theta<s$ and $X:=L^p(\mathbb R^n\times\mathbb R^m)$ with $p\in(2\theta/s, \infty).$ 

Second, let $X:=L_w^{p}\left(\mathbb{R}^{n}\times\mathbb R^m\right)$ with $p\in \left(0,\infty\right)$
and $w\in A_{\infty, R}(\mathbb{R}^{n}\times\mathbb R^m)$. 
Similar to \cite[Remark 2.4(b)]{WangYangYang}, (\ref{e2.6}) holds true when $\theta$, $s\in \left(0,1\right]$, $\theta<s$, $X:=L_w^{p}\left(\mathbb{R}^{n}\times\mathbb R^m\right)$, $p\in\left(\theta,\infty\right)$
and $w\in A_{p/\theta}(\mathbb R^n\times\mathbb R^m)$. Since $\big(L_w^p(\mathbb R^n\times\mathbb R^m)\big)^{s/2}= L_w^{ps/2}(\mathbb R^n\times\mathbb R^m)$ for any $p, s\in(0,\infty)$, we conclude that  (\ref{e2.7}) also holds true when $\theta,s\in(0,1]$, $\theta<s$, $X:=L^p(\mathbb R^n\times\mathbb R^m)$ with $p\in(2\theta/s, \infty)$ and $w\in A_{p/\theta}(\mathbb R^n\times\mathbb R^m)$. 

Third, let $X:=\vec{K}_q^{\alpha, p}\left(\mathbb{R}^n \times \mathbb{R}^m\right)$.
When $0<p<\infty$, $1<q<\infty$ and $\max\{-n/q, -m/q\}<\alpha<\min\{n(1-1/q), m(1-1/q)\}$,
by \cite[Theorem 4.1]{W} the strong maximal operator $\mathcal M_s$ 
is bounded on $X$. From the definition of $\vec{K}_q^{\alpha, p}\left(\mathbb{R}^n \times \mathbb{R}^m\right)$,
we find that 
$$\left\|f^{1/\theta}\right\|^\theta_{\vec{K}_q^{\alpha, p}\left(\mathbb{R}^n \times \mathbb{R}^m\right)}
=\left\|f\right\|_{\vec{K}_{q/\theta}^{\alpha \theta, p/\theta}\left(\mathbb{R}^n \times \mathbb{R}^m\right)}.$$
Then (\ref{e2.6}) holds true for $X:=\vec{K}_q^{\alpha, p}\left(\mathbb{R}^n \times \mathbb{R}^m\right)$ when $\theta$, $s\in \left(0,1\right]$, $\theta<\min\{s, q\}$,
and $$\max\{-n/q, -m/q\}<\alpha\,\theta<\min\{n(1-1/q), m(1-1/q)\}.$$
Similarly, we can see that (\ref{e2.7}) also holds true for $X$
when $\theta$, $s\in \left(0,1\right]$, $\theta<\min\{s, sq/2\}$, and $p\in\left(\theta,\infty\right)$
and $\max\{-n/q, -m/q\}<\frac{2\alpha \theta}{s} <\min\{n(1-1/q), m(1-1/q)\}$.

Moreover, for $X:=\mathcal{M}_{u,p}^{w}\left(\mathbb{R}^{n}\times\mathbb R^m\right)$,
when $1<q_0<p,$ $w\in A_{p/{q_0}, R}(\mathbb R^n\times\mathbb R^m)$
and $u^{q_0}\in \mathbb W_w^{p/q_0}$, from \cite[Theorem 5.3]{W1} we know that
the Fefferman--Stein vector-valued maximal inequalities for $X$ hold true.
From the definition of $\mathcal{M}_{u,p}^{w}\left(\mathbb{R}^{n}\times\mathbb R^m\right)$,
we find that 
$$\left\|f^{1/\theta}\right\|^\theta_{\mathcal{M}_{u,p}^{w}\left(\mathbb{R}^{n}\times\mathbb R^m\right)}
=\left\|f\right\|_{\mathcal{M}_{u^\theta,p/\theta}^{w}\left(\mathbb{R}^{n}\times\mathbb R^m\right)}.$$
Then (\ref{e2.6}) holds true for $X:=\mathcal{M}_{u,p}^{w}\left(\mathbb{R}^{n}\times\mathbb R^m\right)$ when $\theta$, $s\in \left(0,1\right]$, $\theta<s$, $1<q_0<p/\theta<\infty$
and $u^{\theta q_0}\in \mathbb W_w^{p/{q_0}}$. Also, we can similarly show that (\ref{e2.7}) also holds true for $X$.
Last but not least, we can similarly check that $(\ref{e2.6})$ and  $(\ref{e2.7})$ also hold true $X=L^\varphi(\mathbb R^m\times\mathbb R^m)$. We leave the details to the readers.
\end{remark}

\begin{assumption}\label{as:02}
	\textup{Assume that $X$ is a ball quasi-Banach function space, there exists $s\in\left(0,1\right]$ such that $X^{1/s}$ is also a ball Banach function space, and there exist $q\in\left(1,\infty\right]$ and $C\in\left(0,\infty\right)$ such that, for any $f\in\left(X^{1/s}\right)^{\prime}$,}
	\begin{align}\label{eq:newton2}
		\left\|\mathcal M_s^{\left((q / s)^{\prime}\right)}(f)\right\|_{\left(X^{1 / s}\right)^{\prime}} \le C\|f\|_{\left(X^{1 / s}\right)^{\prime}}.
	\end{align}
\end{assumption}

\begin{remark}
It is easy to check that $(\ref{eq:newton2})$ is
equivalent to that 
there exists a constant $C>0$ such that, for any
$f\in[(X^{1/s})']^{1/{(q/s)'}}$,
\begin{align}\label{eq2.8}
		\left\|\mathcal M_s(f)\right\|_{[(X^{1/s})']^{1/{(q/s)'}}} 
		\le C\|f\|_{[(X^{1/s})']^{1/{(q/s)'}}}.
	\end{align}
First, when $X=L^p(\mathbb R^n\times\mathbb R^m)$ for $0<p<\infty$, then for any $0<s<\min\{1,p\}$ and $q\in (\max\{1, p, \infty\})$, we have
$$[(X^{1/s})']^{1/{(q/s)'}}(\mathbb R^n\times\mathbb R^m)=L^{{(p/s)'}/{(q/s)'}}(\mathbb R^n\times\mathbb R^m).$$
Thus by the $L^r(\mathbb R^n\times\mathbb R^m)-$boundedness of $\mathcal M_s$, we conclude that in this case
$(\ref{eq2.8})$ holds true.

Second, when $X=L_w^p(\mathbb R^n\times\mathbb R^m)$ for $0<p<\infty$ and $w\in A_{\infty, R}(\mathbb R^n\times\mathbb R^m)$, then for any $0<s<\min\{1,p\}$, $w\in A_{p/s}(\mathbb R^n\times\mathbb R^m)$ and $q\in (\max\{1, p, \infty\})$
large enough such that $w^{1-(p/s)'}\in A_{(p/s)'/{(q/s)'},R}(\mathbb R^n\times
\mathbb R^m)$ 
, we have
$$[(X^{1/s})']^{1/{(q/s)'}}(\mathbb R^n\times\mathbb R^m)=L_{w^{1-(p/s)'}}^{{(p/s)'}/{(q/s)'}}(\mathbb R^n\times\mathbb R^m).$$
Thus by the weighted Lebesuge boundedness of $\mathcal M_s$, we conclude that in this weighted product case
$(\ref{eq2.8})$ holds true.

Furthermore, when $X$ is the product Herz space $\vec{K}_q^{\alpha, p}\left(\mathbb{R}^n \times \mathbb{R}^m\right)$ with
$0<p<\infty$, $0<q<\infty$ and $\max\{-n/q, -m/q\}<\alpha<\infty$.
Then for any $0<s<\min\{p,q\}$, $\max\{1,q\}<r<\infty$ with
$q/s>1,$ $p/s>1$ and $(q/s)'/(r/s)'>1$, then by duality we conclude that
$$[(X^{1/s})']^{1/{(r/s)'}}(\mathbb R^n\times\mathbb R^m)=\vec{K}_{(q/s)'/(r/s)'}^{-\alpha s(r/s)', (p/s)'/(r/s)'}\left(\mathbb{R}^n \times \mathbb{R}^m\right).$$
Hence, by the fact that $\mathcal M_s$ is bounded on $\vec{K}_q^{\alpha, p}\left(\mathbb{R}^n \times \mathbb{R}^m\right)$ for $1<q<\infty$, we find that in this product Herz case
$(\ref{eq2.8})$ holds true.
Similarly, due to the
fact that the strong maximal operator $\mathcal M_s$ is bounded on
the weighted product Morrey space $\mathcal{M}_{u,p}^{w}\left(\mathbb{R}^{n}\times\mathbb R^m\right)$ (\cite{Ho17, W1}) and its dual spaces as well as
on the he product Musielak--Orlicz--type space $L^\varphi(\mathbb R^n\times \mathbb R^m)$ (\cite{FHLY}),
under the mild assumptions we know that in these cases 
$(\ref{eq2.8})$ still holds true.
\textup{Therefore, we conclude that these spaces satisfy both Assumption \ref{as:01} and Assumption \ref{as:02}.}
\end{remark}

\section{Product Hardy space associated with ball quasi-Banach function spaces}
Let $\psi^{(1)}$ be a Schwartz function on $\mathbb R^n$ which
satisfies
$$
\int_{\mathbb R^n}\psi^{(1)}(x)x^\alpha dx=0,
$$
for all multi-indices $\alpha$,
and for all $\xi\neq 0$
$$
\sum_{k\in\mathbb Z}|\widehat{\psi^{(i)}}(2^{-k}\xi)|^2=1.
$$
Also, $\psi^{(2)}$ is a Schwartz function on $\mathbb R^m$ which satisfies the similar conditions.
For $f\in({\mathcal{S}}_M)'(\mathbb R^n\times\mathbb R^m)$,
define the Littlewood-Paley square function of $f$ by
$$
\mathcal G(f)(x):=\bigg\{\sum_{j\in\mathbb Z}\sum_{k\in\mathbb Z}|\psi_{j,k}\ast f(x,y)|^2\bigg\}^{1/2},
$$
where $\psi_{j,k}=\psi_j^{(1)}\otimes\psi_k^{(2)}$, 
$\psi_j(\cdot)=2^{kn}\psi(2^{j}\cdot)$ and $\psi_k(\cdot)=2^{km}\psi(2^{k}\cdot)$.
Now we define the product Hardy spaces $H_X(\mathbb R^n\times\mathbb R^m)$.
\begin{definition}\label{s3d1}
Let $X(\mathbb R^n\times \mathbb R^m)$ be a ball quasi-Banach function space.
The product Hardy spaces $H_X(\mathbb R^n\times \mathbb R^m)$ are the collection of all $f\in({\mathcal{S}}_M)'(\mathbb R^n\times\mathbb R^m)$
for which the quantity
$$\|f\|_{{H}_X(\mathbb R^n\times\mathbb R^m)}=\|\mathcal G(f)\|_{X(\mathbb R^{n+m})}<\infty.$$
\end{definition}

In order to prove that the definition of the new product Hardy spaces is independent of any particular choice of
$\psi_{j,k}$ and give the equivalent characterization in terms of the discrete Littlewood-Paley square function,
the new Plancherel-P\'olya type inequalities play a crucial role.
To prove it,
we recall the following discrete Calder\'on reproducing formula in \cite[Theorem 3.2]{DHLW}.

\begin{proposition}\label{DCRF}  Suppose that $\varphi_{j, k}$ are the same as above. Then for any $M \geq 1$, we can choose a large $N$ depending on $M$ and $\varphi$ such that the following discrete Calderón's identity
$$
f(x, y)=\sum_{j, k, J, I}|I||J| \tilde{\varphi}_{j, k}\left(x, y, x_I, y_J\right) \varphi_{j, k} * f\left(x_I, y_J\right)
$$
holds in $\mathcal{S}_M\left(\mathbb{R}^n \times \mathbb{R}^m\right)$ and in the dual space $\left(\mathcal{S}_M\right)^{\prime}\left(\mathbb{R}^n \times \mathbb{R}^m\right)$, where $\tilde{\varphi}_{j, k}(x, y$, $\left.x_I, y_J\right) \in \mathcal{S}_M\left(\mathbb{R}^n \times \mathbb{R}^m\right),$ $I, J$ are dyadic cubes with side-length $l(I)=2^{-j-N}$ and $l(J)=$ $2^{-k-N}$ and $x_I, y_J$ are any fixed points in $I, J$ respectively.
\end{proposition}

Now applying the discrete Calder\'on reproducing formula provides the following Plancherel-P\'olya type inequalities. 

\begin{theorem}\label{Planc} Suppose $\psi_{j, k}, \varphi_{j, k}$ satisfy the same conditions as above. 
Assume that $X(\mathbb R^n\times\mathbb R^m)$ is a ball quasi-Banach function space satisfying Assumption \ref{as:01} with some
$0<\theta<s\le1$ and $M\gg 1$. Then for $f \in\left(\mathcal{S}_M\right)^{\prime}\left(\mathbb{R}^n \times \mathbb{R}^m\right)$,
$$
\begin{aligned}
&\left\|\left\{\sum_{j, k} \sum_{I, J} \sup _{u \in I, v \in J}\left|\psi_{j, k}\ast f(u, v)\right|^2 \chi_I \chi_J\right\}^{\frac{1}{2}}\right\|_{X\left(\mathbb{R}^n \times \mathbb{R}^m\right)} \\
&\quad \sim\left\|\left\{\sum_{j, k} \sum_{I, J} \inf _{u \in I, v \in J}\left|\varphi_{j, k}\ast f(u, v)\right|^2 \chi_I \chi_J\right\}^{\frac{1}{2}}\right\|_{X(\mathbb{R}^n \times \mathbb{R}^m)},
\end{aligned}
$$
where $I, J$ are dyadic cubes with side-length $l(I)=2^{-j-N}$ and $l(J)=$ $2^{-k-N}$
for some fixed $N$.
\end{theorem}

\begin{proof}
By Proposition \ref{DCRF}, for any $f \in\left(\mathcal{S}_M\right)^{\prime}$ 
we conclude that
$$
f(x, y)=\sum_{j^{\prime}} \sum_{k^{\prime}} \sum_{J^{\prime}} \sum_{I^{\prime}}\left|J^{\prime} \| I^{\prime}\right| \tilde{\varphi}_{j^{\prime}, k^{\prime}}\left(x, y, x_{I^{\prime}}, y_{J^{\prime}}\right)\left(\varphi_{j^{\prime}, k^{\prime}} * f\right)\left(x_{I^{\prime}}, y_{J^{\prime}}\right) 
$$
and that
$$
\left(\psi_{j, k} * f\right)(u, v)=\sum_{J^{\prime}} \sum_{I^{\prime}}\left|J^{\prime} \| I^{\prime}\right|\left(\psi_{j, k} * \tilde{\varphi}_{j^{\prime}, k^{\prime}}\left(\cdot, \cdot, x_{I^{\prime}}, y_{J^{\prime}}\right)\right)(u, v)\left(\varphi_{j^{\prime}, k^{\prime}} * f\right)\left(x_{I^{\prime}}, y_{J^{\prime}}\right) .
$$
Since $\tilde{\varphi} \in \mathcal{S}_M$, by following the same argument 
in \cite[pp. 48]{DHLW}, 
for any $u, u^\ast, x_{I^{\prime}} \in I, v, v^\ast, y_{J^{\prime}} \in J$ we have
$$
\begin{array}{rl}
&|\psi_{j, k}\ast f(u, v)|\\
\leq & C 
\sum_{k^{\prime}, j^{\prime}} \sum_{I^{\prime}, J^{\prime}} 2^{-\left|j-j^{\prime}\right| M} 2^{-\left|k-k^{\prime}\right| M}\left|I^{\prime}\right|\left|J^{\prime}\right| \\
& \times \frac{2^{-\left(j \wedge j^{\prime}\right) M}}{\left(2^{-\left(j \wedge j^{\prime}\right)}+\left|u-x_{I^{\prime}}\right|\right)^{n+M}} \cdot \frac{2^{-\left(k \wedge k^{\prime}\right) M}}{\left(2^{-\left(k \wedge k^{\prime}\right)}+\left|v-y_{J^{\prime}}\right|\right)^{n+M}}\left|\varphi_{j^{\prime}, k^{\prime}} * f\left(x_{I^{\prime}}, y_{J^{\prime}}\right)\right| \\
\leq & C_{r, N, m, n} \sum_{k^{\prime}, j^{\prime}} 2^{-\left|j-j^{\mid}\right| M_1} 2^{-\left|k-k^{\prime}\right| M_2}\left\{\mathcal{M}_s\left(\sum_{J^{\prime}, I^{\prime}}\left|\varphi_{j^{\prime}, k^{\prime}} * f\left(x_{I^{\prime}}, y_{J^{\prime}}\right)\right| \chi_{J^{\prime}} \chi_{I^{\prime}}\right)^r\left(u^*, v^*\right)\right\}^{\frac{1}{r}} \\
\leq & C\left\{\sum_{k^{\prime}, j^{\prime}} 2^{-\left|j-j^{\prime}\right| M_1 2^{-\left|k-k^{\prime}\right| M_2}}\right\}^{1 / 2} \\
&\quad \times\left\{\sum_{k^{\prime}, j^{\prime}} 2^{-\left|j-j^{\prime}\right| M_1} 2^{-\left|k-k^{\prime}\right| M_2}\left\{\mathcal{M}_s\left(\sum_{J^{\prime}, I^{\prime}}\left|\varphi_{j^{\prime}, k^{\prime}} * f\left(x_{I^{\prime}}, y_{J^{\prime}}\right)\right| \chi_{J^{\prime}} \chi_{I^{\prime}}\right)^r\left(u^*, v^*\right)\right\}^{\frac{2}{r}}\right\}^{1 / 2} \\
\leq & C\left\{\sum_{k^{\prime}, j^{\prime}} 2^{-\left|j-j^{\prime}\right| M_1} 2^{-\left|k-k^{\prime}\right| M_2}\left\{\mathcal{M}_s\left(\sum_{J^{\prime}, I^{\prime}}\left|\varphi_{j^{\prime}, k^{\prime}} * f\left(x_{I^{\prime}}, y_{J^{\prime}}\right)\right| \chi_{J^{\prime}} \chi_{I^{\prime}}\right)^r\left(u^*, v^*\right)\right\}^{\frac{2}{r}}\right\}^{1 / 2}
\end{array}
$$
where $M_1=M-(1 / r-1) n$ and $M_2=M-(1 / r-1) m$ are positive constants. 
Fix $M\gg 1$. 
Denote $r:=\frac{2\theta}{s}$, where $\theta$ and $s$ are as in Assumptions \ref{as:01}.
Thus, for any $r$ fulfilling $\max \left\{\frac{n}{n+M}, \frac{m}{m+M}\right\}<r<1$, we get that
$$
\begin{aligned}
&\left\{\sum_{j, k} \sum_{I, J} \sup _{u \in I, v \in J}\left|\psi_{j, k} * f(u, v)\right|^2 \chi_I(x) \chi_J(y)\right\}^{\frac{1}{2}} \\
&\quad \leq C\left\{\sum_{j^{\prime}, k^{\prime}}\left\{\mathcal{M}_s\left(\sum_{I^{\prime}, J^{\prime}} \inf _{u^{\prime} \in I^{\prime}, v^{\prime} \in J^{\prime}}\left|\varphi_{j^{\prime}, k^{\prime}} * f\left(u^{\prime}, v^{\prime}\right)\right| \chi_{J^{\prime}} \chi_{I^{\prime}}\right)^r(x, y)\right\}^{2 / r}\right\}^{\frac{1}{2}},
\end{aligned}
$$
where $I \subset \mathbb{R}^n, J \in \mathbb{R}^m$ are dyadic cubes with side-length $2^{-j-N}$ and $2^{-k-N}$ respectively for the above $N$. 
From this and Assumption \ref{as:01} of $\mathcal{M}_s$, 
\begin{align*}
&\left\|\left\{\sum_{j^{\prime}, k^{\prime}}\left\{\mathcal{M}_s\left(\sum_{I^{\prime}, J^{\prime}} \inf _{u^{\prime} \in I^{\prime}, v^{\prime} \in J^{\prime}}\left|\varphi_{j^{\prime}, k^{\prime}} * f\left(u^{\prime}, v^{\prime}\right)\right| \chi_{J^{\prime}} \chi_{I^{\prime}}\right)^r\right\}^{\frac{2}{r}}\right\}^{\frac{1}{2}}
\right\|_{X(\mathbb R^n\times\mathbb R^m)}\\
&= C
\left\|\left\{\sum_{j^{\prime}, k^{\prime}}\left\{\mathcal{M}^\theta_s\left(\sum_{I^{\prime}, J^{\prime}} \inf _{u^{\prime} \in I^{\prime}, v^{\prime} \in J^{\prime}}\left|\varphi_{j^{\prime}, k^{\prime}} * f\left(u^{\prime}, v^{\prime}\right)\right| \chi_{J^{\prime}} \chi_{I^{\prime}}\right)^\frac{2}{s}\right\}^{s}\right\}^{\frac{1}{s}}
\right\|^{\frac{s}{2}}_{X^{\frac{s}{2}}(\mathbb R^n\times\mathbb R^m)}\\
&\le C
\left\|\left\{\sum_{j^{\prime}, k^{\prime}}\left(\sum_{I^{\prime}, J^{\prime}} \inf _{u^{\prime} \in I^{\prime}, v^{\prime} \in J^{\prime}}\left|\varphi_{j^{\prime}, k^{\prime}} * f\left(u^{\prime}, v^{\prime}\right)\right| \chi_{J^{\prime}} \chi_{I^{\prime}}\right)^{2}\right\}^{\frac{1}{s}}
\right\|^{\frac{s}{2}}_{X^{\frac{s}{2}}(\mathbb R^n\times\mathbb R^m)}\\
&\leq C\left\|\left\{\sum_{j^{\prime}, k^{\prime}} \sum_{I^{\prime}, J^{\prime}} \inf _{u^{\prime} \in I^{\prime}, v^{\prime} \in J^{\prime}}\left|\varphi_{j^{\prime}, k^{\prime}}\ast f(u^{\prime}, v^{\prime})\right|^2 \chi_{I^{\prime}} \chi_{J^{\prime}}\right\}^{\frac{1}{2}}\right\|_{X(\mathbb{R}^n \times \mathbb{R}^m)}.
\end{align*}
The proof of the inverse inequality is identical.
Thus, we completes the proof of Theorem \ref{Planc}.
\end{proof}

As products, we find that the new product Hardy spaces is well defined and give the following discrete Littlewood--Paley--Stein characterizations for these new spaces. 
Assume that $\phi^{(1)}\in C_c^\infty(\mathbb R^n)$
satisfying
$$
\int_{\mathbb R^n}\phi^{(1)}(x)x^\alpha dx=0,
$$
for all multi-indices $\alpha$ fulfilling $0\le |\alpha|\le M_0$,
where $M_0>0$ is a sufficiently large integer
and for all $\xi_1\neq 0$
$$
\sum_{k\in\mathbb Z}|\widehat{\phi^{(i)}}(2^{-k}\xi_1)|^2=1.
$$
Moreover, $\phi^{(2)}$ is a Schwartz function on $\mathbb R^m$ which satisfies the similar conditions. Similarly, define the Littlewood-Paley square function associated with $\phi$ of $f$ by
$$
\mathcal G_\phi(f)(x):=\bigg\{\sum_{j\in\mathbb Z}\sum_{k\in\mathbb Z}|\phi_{j,k}\ast f(x,y)|^2\bigg\}^{1/2}.
$$

\begin{corollary}
Assume that $X(\mathbb R^n\times\mathbb R^m)$ is a ball quasi-Banach function space satisfying Assumption \ref{as:01} with some
$0<\theta<s\le1$ and $M\gg 1$. Let $f\in({\mathcal{S}}_M)'(\mathbb R^n\times\mathbb R^m)$.
Then $f\in H_X(\mathbb R^n\times\mathbb R^m)$
if and only if $\mathcal G^d(f)\in X(\mathbb R^n\times\mathbb R^m)$.
Furthermore,
$$\|f\|_{{H}_X(\mathbb R^n\times\mathbb R^m)}\sim
\|\mathcal G_\phi(f)\|_{X(\mathbb R^{n+m})}\sim
\|\mathcal G^d(f)\|_{X(\mathbb R^{n+m})}\sim
\|\mathcal G_\phi^d(f)\|_{X(\mathbb R^{n+m})}<\infty,$$
where
$$
\mathcal G^d(f)(x):=\bigg\{\sum_{j\in\mathbb Z}\sum_{k\in\mathbb Z}
|\psi_{j,k}\ast f(x_I,y_I)|^2\chi_I(x)\chi_J(y)\bigg\}^{1/2}
$$
and
$$
\mathcal G_\phi^d(f)(x):=\bigg\{\sum_{j\in\mathbb Z}\sum_{k\in\mathbb Z}
|\phi_{j,k}\ast f(x_I,y_I)|^2\chi_I(x)\chi_J(y)\bigg\}^{1/2}.
$$
\end{corollary}

\begin{remark}
The Hardy type spaces obtained in this section have wide generality and applications. Here we present some concrete examples of ball quasi-Banach function spaces $X(\mathbb R^n\times \mathbb R^m)$ and the corresponding Hardy type spaces $H_X(\mathbb R^n\times \mathbb R^m)$.
When the product ball quasi-Banach function spaces $X$ is the
classical Lebesgue space $L^{p}\left(\mathbb{R}^{n}\times\mathbb R^m\right)$, the weighted Lebesgue space associated with product Muckenhoupt weights $L_w^{p}\left(\mathbb{R}^{n}\times\mathbb R^m\right)$, the product Herz space $\vec{K}_q^{\alpha, p}\left(\mathbb{R}^n \times \mathbb{R}^m\right)$, the weighted product Morrey space $\mathcal{M}_{u,p}^{w}\left(\mathbb{R}^{n}\times\mathbb R^m\right)$ and the product Musielak--Orlicz--type space $L^\varphi(\mathbb R^m\times\mathbb R^m)$, respectively, then $H_X(\mathbb R^n\times \mathbb R^m)$
is the classical product Hardy space $H^{p}\left(\mathbb{R}^{n}\times\mathbb R^m\right)$, the weighted product Hardy space associated with product Muckenhoupt weights $H_w^{p}\left(\mathbb{R}^{n}\times\mathbb R^m\right)$, the product Herz--Hardy space  $H\vec{K}_q^{\alpha, p}\left(\mathbb{R}^n \times \mathbb{R}^m\right)$, the weighted product Hardy--Morrey space $H\mathcal{M}_{u,p}^{w}\left(\mathbb{R}^{n}\times\mathbb R^m\right)$ and the product Musielak--Orlicz--Hardy--type space $H^\varphi(\mathbb R^n\times\mathbb R^m)$. 
Besides these spaces, we also deem that our methods for defining product Hardy type spaces also suitable for
grand product Hardy spaces,
grand product Hardy-Morrey spaces.
\end{remark}

When we consider the boundedness of operators, the ``so-called'' density argument is very useful. Thus, we obtain the product test function space $\mathcal S_M(\mathbb R^n\times\mathbb R^m)$ is dense in the new product Hardy space with the  additional assumption of absolute continuity of the quasi-norm. We point out that, except the weighted product Morrey space, the other examples of ball quasi-Banach function spaces in our paper all have absolutely continuous quasi-norms.

\begin{definition}
\textup{Let $X(\mathbb R^n\times\mathbb R^m)$ be a ball product quasi-Banach function space. A function $f\in X(\mathbb R^n\times\mathbb R^m)$ is said to have an absolutely continuous quasi-norm in $X(\mathbb R^n\times\mathbb R^m)$ if
$\|f\chi_{E_j}\|_{X(\mathbb R^n\times\mathbb R^m)}\downarrow 0$
whenever $\{E_j\}_{j=1}^{\infty}$ is a sequence of measurable sets that satisfy
$E_{j+1}\subset E_j$ for any $j\in\mathbb N$ and $\bigcap_{j=1}^\infty E_j=\emptyset$.
Moreover, $X(\mathbb R^n\times\mathbb R^m)$ is said to have an absolutely
continuous quas-norm if, for any $f\in X(\mathbb R^n\times\mathbb R^m)$, $f$ has an absolutely continuous
quasi-norm in $X(\mathbb R^n\times\mathbb R^m)$.}
\end{definition}

\begin{theorem}\label{dense}
Let $X(\mathbb R^n\times\mathbb R^m)$ is a ball quasi-Banach function space satisfying Assumption \ref{as:01} with some
$0<\theta<s\le1$ and $M\gg 1$. Assume further that $X(\mathbb R^n\times\mathbb R^m)$ has an absolutely
continuous quasi-norm. Then the product test function space $\mathcal S_M(\mathbb R^n\times\mathbb R^m)$ is dense ${H}_X(\mathbb R^n\times\mathbb R^m)$.
\end{theorem}

\begin{proof}
Set $W=\{(j, k, I, J):|j| \leq L,|k| \leq L, I \times$ $J \subset B(0, r)\}$, where $I, J$ are dyadic cubes in $\mathbb{R}^n, \mathbb{R}^m$ with side length $2^{-j-N}, 2^{-k-N}$, respectively. 
If $f \in H_X\left(\mathbb{R}^n \times \mathbb{R}^m\right)$, then by Proposition \ref{DCRF} we conclude that
$$
f=\sum_{j, k, I, J}|I||J| \tilde{\psi}_{j, k}\left(x, y, x_I, y_J\right) \psi_{j, k} * f\left(x_I, y_J\right) \qquad\mbox{in}\quad \left(\mathcal{S}_M\right)^{\prime}\left(\mathbb{R}^n \times \mathbb{R}^m\right).
$$
By following the similar argument of Theorem \ref{Planc}, for any $f \in H_X\left(\mathbb{R}^n \times \mathbb{R}^m\right)$, it yields that
$$
\begin{aligned}
&\left\| f -\sum_{(j, k, I, J) \in W}|I||J| \tilde{\psi}_{j, k}\left(\cdot, \cdot, x_I, y_J\right) \psi_{j, k} * f\left(x_I, y_J\right) \right\|_{H_X\left(\mathbb{R}^n \times \mathbb{R}^m\right)} \\
&=\left\|\sum_{(j, k, I, J) \in W^c}|I||J| \tilde{\psi}_{j, k}\left(\cdot, \cdot, x_I, y_J\right) \psi_{j, k} * f\left(x_I, y_J\right)\right\|_{H_X\left(\mathbb{R}^n \times \mathbb{R}^m\right)} \\
& \leq C\left\|\left\{\sum_{(j, k, I, J) \in W^c}\left|\psi_{j, k} * f\left(x_I, y_J\right)\right|^2 \chi_I \chi_J\right\}\right\|_{X\left(\mathbb{R}^{n+m}\right)}\rightarrow 0,
\end{aligned}
$$
as $L, r$ tend to infinity, where the last inequality comes from the fact that $X$ has an absolutely
continuous quasi-norm. Due to the following term
$$
\sum_{(j, k, I, J) \in W}|I||J| \tilde{\psi}_{j, k}\left(x, y, x_I, y_J\right) \psi_{j, k} * f\left(x_I, y_J\right)
$$  belongs to $\mathcal{S}_M\left(\mathbb{R}^n \times \mathbb{R}^m\right)$
and then we have completed the proof of Theorem \ref{dense}.
\end{proof}

\begin{corollary}
Let $X(\mathbb R^n\times\mathbb R^m)$ is a ball quasi-Banach function space satisfying Assumption \ref{as:01} with some
$0<\theta<s\le1$ and $M\gg 1$. Assume further that $X$ has an absolutely
continuous quasi-norm. Then ${H}_X(\mathbb R^n\times\mathbb R^m)\cap L^2(\mathbb R^{n+m})$ is dense ${H}_X(\mathbb R^n\times\mathbb R^m)$.
\end{corollary}

\section{Decomposition for product Hardy spaces}
In this section, we will establish a new decomposition for product Hardy spaces associated with ball quasi-Banach function spaces ${H}_X(\mathbb R^n\times\mathbb R^m)$. For our purpose, the following discrete type Calder\'on identity associated with $\phi$ is needed, which plays a key role in the proof of decomposition for ${H}_X(\mathbb R^n\times\mathbb R^m)$.
\begin{theorem}\label{DTCI}
Let $X(\mathbb R^n\times\mathbb R^m)$ is a ball quasi-Banach function space satisfying Assumption \ref{as:01} with some
$0<\theta<s\le1$. For any $1<q<\infty$,
then there exists an operator $T_N^{-1}$ such that
$$
f(x, y)=\sum_{j, k} \sum_{I, J}|I||J| \phi_{j, k}\left(x-x_I, y-y_J\right) \phi_{j, k} *\left(T_N^{-1}(f)\right)\left(x_I, y_J\right)
$$
where $T_N^{-1}$ is bounded on $L^q\left(\mathbb{R}^{n+m}\right)$ and $H_X\left(\mathbb{R}^n \times \mathbb{R}^m\right)$, and the series converges in $L^q\left(\mathbb{R}^{n+m}\right)$.
\end{theorem}

\begin{proof} 
From \cite[Theorem 3.4]{DHLW}, we know that
the series converges in $L^q\left(\mathbb{R}^{n+m}\right)$
and $T_N^{-1}$ is bounded on $L^q\left(\mathbb{R}^{n+m}\right)$.
To achieve our goal, we only need to prove that $T_N^{-1}$ is also bounded on $H_X\left(\mathbb{R}^n \times \mathbb{R}^m\right)$.
First observe that
$$
\begin{aligned}
f(x, y) &=\sum_{j, k, I, J} \int_I \int_J \phi_{j, k}(x-u, y-w)\left(\phi_{j, k} * f\right)(u, w) d u d w \\
&=:\sum_{j, k} \sum_{I, J}|I||J| \phi_{j, k}\left(x-x_I, y-y_J\right)\left(\phi_{j, k} * f\right)\left(x_I, y_J\right)+\mathcal{R}(f)(x, y) .
\end{aligned}
$$
where $\mathcal R$ is the remainder operator.
Then we conclude that
$$
\|\mathcal{R}(f)\|_{H_X\left(\mathbb{R}^n \times \mathbb{R}^m\right)} \leq C 2^{-N}\|f\|_{H_X\left(\mathbb{R}^n \times \mathbb{R}^m\right)} .
$$
Indeed, by using Proposition \ref{DCRF}, the almost orthogonality estimate and the similar argument as in the proof of Theorem \ref{Planc} we have
\begin{align*}
&\|\mathcal{R}(f)\|_{H_X\left(\mathbb{R}^n \times \mathbb{R}^m\right)}=\|\mathcal{G}(\mathcal{R}(f))\|_{X(\mathbb{R}^n \times \mathbb{R}^m)} \\
&\leq C\left\|\left\{\sum_{j, k} \sum_{I, J}\left|\psi_{j, k} * \mathcal{R}(f)\right|^2 \chi_I \chi_J\right\}^{1 / 2}\right\|_{X(\mathbb{R}^n \times \mathbb{R}^m)} \\
&=C\left\|\left\{\left.\sum_{j, k, I, J} \sum_{y^{\prime}, k^{\prime}, r^{\prime}, J^{\prime}}\left|J^{\prime}\right|\left\|I^{\prime}\right\| \psi_{j, k} * \mathcal{R}\left(\tilde{\psi}_{p^{\prime}, k^{\prime}}\left(\cdot, x_{I^{\prime}}, \cdot, J^{\prime}\right) \cdot \psi_{j^{\prime}, k^{\prime}} * f\left(x_{I^{\prime}}, y_{J^{\prime}}\right)\right)\right|^2 \chi_I \chi_J\right\}^{1 / 2}\right\|_{X(\mathbb{R}^n \times \mathbb{R}^m)}\\
& \leq C 2^{-N}\left\|\left\{\sum_{\gamma, k^{\prime}}\left\{\mathcal{M}_s\left(\sum_{J^{\prime}, l^{\prime}}\left|\psi_{j, k^{\prime}} * f\left(x_{I^{\prime}}, y_{J^{\prime}}\right)\right| \chi_{J^{\prime}} \chi_{I^{\prime}}\right)^r\right\}^{\frac{2}{r}}\right\}^{\frac{1}{2}}\right\|_{X(\mathbb{R}^n \times \mathbb{R}^m)} \\
& \leq C 2^{-N}\left\|\left\{\sum_{j^{\prime}, k^{\prime}, J^{\prime}, I^{\prime}}\left|\psi_{\prime^{\prime}, k^{\prime}} * f\left(x_{I^{\prime}}, y_{J^{\prime}}\right)\right|^2 \chi_{I^{\prime}} \chi_{J^{\prime}}\right\}^{\frac{1}{2}}\right\|_{X(\mathbb{R}^n \times \mathbb{R}^m)} \leq C 2^{-N}\|f\|_{H_X\left(\mathbb{R}^n \times \mathbb{R}^m\right)} .
\end{align*}
where $j, k, \psi, \chi_I, \chi_J, x_I, y_J$ are the same as in Theorem \ref{Planc}.
For our purpose, we write $\left(T_N\right)^{-1}=\sum_{i=0}^{\infty} \mathcal R^i$, where
$$
T_N(f)=\sum_{j, k, J, I}|I||J| \phi_{j, k}\left(x-x_I, y-y_J\right)\left(\phi_{j, k} * f\right)\left(x_I, y_J\right) .
$$
From the fact that if $N$ is large enough, then both of $T_N$ and $\left(T_N\right)^{-1}$ are bounded on $H_X(\mathbb{R}^n\times\mathbb{R}^m)$ and $L^q\left(\mathbb{R}^{n+m}\right)$. Therefore, this completes the proof.
\end{proof}

Let $X(\mathbb R^n\times\mathbb R^m)$ is a ball quasi-Banach function space and $2\leq q<\infty$.
A measurable function $b$ is said to be an $(X,q)-$product atom,
if there exists a measurable open set $\Omega$ of $\mathbb R^{n+m}$ with finite measure such that $\mbox{supp}(b)\subset \Omega$ and 
$$\|b\|_{L^q(\mathbb R^{n+m})}\le C\frac{|\Omega|^{\frac{1}{q}}}{\|\chi_{\Omega}\|_{X(\mathbb R^{n+m})}}.$$
Furthermore, $b$ can be decomposed into a rectangle $(X,q)-$atom $b_R$ associated
to the rectangle $R=I\times J$ which is supported in $3R$ with
$$
b=\sum_{R\in m(\Omega)}b_R;\quad \left(\sum_{R\in m(\Omega)}\|b_R\|^q_{L^q(\mathbb R^{n+m})}\right)^{\frac{1}{q}}
\le \frac{|\Omega|^{\frac{1}{q}}}{\|\chi_\Omega\|_{X(\mathbb R^{n+m})}},
$$
where $m(\Omega)$ is the set of all maximal rectangles contained in $\Omega$, $b_R$ is supported
in $3R$ and
for any $0\le |\alpha|,|\beta|\le N_0$,
$$
\int_{\mathbb R^n}b_R(x,y)x^\alpha dx=0, \; a.e.\;y\in R^m;
\quad\quad 
\int_{\mathbb R^m}b_R(x,y)y^\alpha dy=0, \; a.e.\;x\in R^n.
$$

Now we state the main result in this section as follows.
\begin{theorem}\label{Decom}
Let $X(\mathbb R^n\times\mathbb R^m)$ is a ball quasi-Banach function space satisfying Assumption \ref{as:01} with some
$0<\theta<s\le1$.
Assume further that $X(\mathbb R^n\times\mathbb R^m)$ has an absolutely
continuous quasi-norm and $2\leq q<\infty$. If $f\in H_X(\mathbb R^n\times\mathbb R^m)\cap L^q(\mathbb R^{n+m})$, 
then there exists countable collections of of $(X,q)-$product atom $\{b_\ell\}_\ell$ of
$H_X(\mathbb R^n\times\mathbb R^m)$ and
of non-negative numbers $\{\lambda_\ell\}_\ell$
such that
$$
f(x, y)=\sum_{\ell}\lambda_\ell b_\ell(x,y),
$$
and that
$$
\left\|\left\{\sum_{\ell=1}^\infty\left(\frac{\lambda_\ell}{\|\chi_{\Omega_\ell\|_{X(\mathbb R^{(n+m)}}}}\right)^s\chi_{\Omega_\ell}\right\}^{\frac{1}{s}}\right\|_{X(\mathbb R^{n+m})}
\le C\|f\|_{H_X(\mathbb R^n\times\mathbb R^m)}
$$
where the series holds in $L^q\left(\mathbb{R}^{n+m}\right)$ and $H_X\left(\mathbb{R}^n \times \mathbb{R}^m\right)$.
\end{theorem}

\begin{proof}
Define a new discrete Littlewood--Paley--Stein function by
$$
\tilde{g}(f)(x, y)=\left\{\sum_{j, k} \sum_{I, J}\left|\phi_{j, k} *\left(T_N^{-1}(f)\right)\left(x_I, y_J\right)\right|^2 \chi_I(x) \chi_J(y)\right\}^{1 / 2} .
$$
By following the similar argument to the proof of Theorem \ref{Planc}, for any $f \in H_X\left(\mathbb{R}^n \times \mathbb{R}^m\right)$, we conclude that
$$
 \|f\|_{H_X\left(\mathbb{R}^n \times \mathbb{R}^m\right)}\sim \|\tilde{g}(f)\|_{X\left(\mathbb{R}^{n+m}\right)}.
$$
Suppose that $f \in H_X\left(\mathbb{R}^n \times \mathbb{R}^m\right)\cap L^q(\mathbb R^{n+m})$. Set
$$
\Omega_\ell=\left\{(x, y) \in \mathbb{R}^n \times \mathbb{R}^m: \tilde{g}(f)(x, y)>2^\ell\right\}
$$
Denote
$$
B_\ell=\left\{R=I \times J: \left|R \cap \Omega_\ell\right|>\frac{1}{2}|R|, \left|R \cap \Omega_{\ell+1}\right| \leq \frac{1}{2} |R|\right\}
$$
where $I, J$ are dyadic cubes in $\mathbb{R}^n, \mathbb{R}^m$ with side-length $2^{-j-N}, 2^{-k-N}$ respectively.
For short, we denote that $\phi_R:=\phi_{j, k}$, where $R=I \times J$. By Proposition \ref{DCRF}, we rewrite
$$
f(x, y)=\sum_\ell \sum_{R \in B_\ell}|R| \tilde{\phi}_R\left(x-x_I, y-y_J\right) \phi_R *\left(T_N^{-1}(f)\right)\left(x_I, y_J\right)
$$
where the series converges in $L^q$ norm.
Define an open set $\tilde\Omega_\ell$ by
$$
\tilde\Omega_\ell:=\{(x,y):\mathcal M_s(\chi_{\Omega_\ell})(x,y)>\frac{1}{1000}\}.
$$
 To establish the decomposition, we then rewrite
$$
f(x, y)=\sum_{\ell}\lambda_\ell b_\ell(x,y),
$$
where $$\lambda_\ell=C2^{\ell}
\|\chi_{\tilde\Omega_\ell}\|_{X(\mathbb R^{n+m})}
$$
and
$$
b_\ell(x,y)=\frac{1}{\lambda_\ell}\sum_{R \in B_\ell}|R| \tilde{\phi}_R\left(x-x_I, y-y_J\right) \phi_R *\left(T_N^{-1}(f)\right)\left(x_I, y_J\right).
$$
If $R\in B_\ell$, then $\phi_R$ is supported in $\tilde\Omega_\ell$. This implies that
$\mbox{supp}(b_\ell)\subset\tilde\Omega_\ell$.
Applying the duality theorem, the Cauchy--Schwarz inequality and the H\"older inequality, it yields that
\begin{align*}
&\left\|\sum_{R \in B_\ell}|R| \widetilde{\phi}_R\left(x-x_I, y-y_J\right) \phi_R *\left(T_N^{-1}(f)\right)\left(x_I, y_J\right)\right\|_{L^q(\mathbb R^{n+m})}\\
=& \sup_{\|h\|_{L^{q'}(\mathbb R^{n+m})}\le 1}\left|\left\langle\sum_{R \in B_\ell}|R|\phi_R\left(x-x_I, y-y_J\right) \phi_R\ast\left(T_N^{-1}(f)\right)\left(x_I, y_J\right), h\right\rangle\right| 
\\
=&\sup_{\|h\|_{L^{q'}(\mathbb R^{n+m})}\le 1}\left|\sum_{R \in B_\ell}|R|\bar{\phi}_R \ast h\left(x_I, y_J\right) \phi_R *\left(T_N^{-1}(f)\right)\left(x_I, y_J\right)\right| \\
=& \sup_{\|h\|_{L^{q'}(\mathbb R^{n+m})}\le 1}\left|\sum_{R \in B_\ell} \int_{\mathbb R^n\times\mathbb R^m} \bar{\phi}_R\ast h\left(x_I, y_J\right) \phi_R\ast\left(T_N^{-1}(f)\right)\left(x_I, y_J\right) \chi_R(x, y) dxdy\right| \\
\leq &\sup_{\|h\|_{L^{q'}(\mathbb R^{n+m})}\le 1}\left(\int_{\mathbb R^n\times \mathbb R^m}\left(\sum_{R \in B_\ell}\left|\phi_R\ast\left(T_N^{-1}(f)\right)\left(x_I, y_J\right)\right|^2 \chi_R(x, y)\right)^{\frac{q}{2}}dxdy\right)^{\frac{1}{q}} \\
& \times\left(\int_{\mathbb R^n\times\mathbb R^m}\left(\sum_{R \in B_\ell}\left|\bar{\phi}_R\ast h\left(x_I, y_J\right)\right|^2 \chi_R(x, y)\right)^{\frac{q'}{2}}dxdy\right)^{\frac{1}{q'}} \\
\leq &C\left(\int_{\mathbb R^n\times \mathbb R^m}\left(\sum_{R \in B_\ell}\left|\phi_R\ast\left(T_N^{-1}(f)\right)\left(x_I, y_J\right)\right|^2 \chi_R(x, y)\right)^{\frac{q}{2}}dxdy\right)^{\frac{1}{q}} 
\end{align*}
where $\bar{\phi}_{j, k}(x, y)=\phi_{j, k}(-x,-y)$.
To end it, by applying the classical Fefferman--Stein vector valued inequality, we conclude that
\begin{align*}
&\left(\int_{\mathbb R^n\times \mathbb R^m}\left(\sum_{R \in B_\ell}\left|\phi_R\ast\left(T_N^{-1}(f)\right)\left(x_I, y_J\right)\right|^2 \chi_R(x, y)dxdy\right)^{\frac{q}{2}}\right)^{\frac{1}{q}} \\
=&\left\|\left\{\sum_{R \in B_\ell}
\left|\phi_R\ast\left(T_N^{-1}(f)\right)\left(x_I, y_J\right)\right|^2 \chi_R\right\}^{\frac{1}{2}}\right\|_{L^q(\mathbb R^{n+m})}\\
\leq &C\left\|\left\{\sum_{R \in B_\ell}
\left|\phi_R\ast\left(T_N^{-1}(f)\right)\left(x_I, y_J\right)\right|^2 \mathcal{M}_s^2\left(\chi_{R \cap \tilde{\Omega}_\ell \backslash \Omega_{\ell+1}}\right)\right\}^{\frac{1}{2}}\right\|_{L^q(\mathbb R^{n+m})}\\
\leq &C\left\|\left\{\sum_{R \in B_\ell}
\left|\phi_R\ast\left(T_N^{-1}(f)\right)\left(x_I, y_J\right)\right|^2 \chi_{R \cap \tilde{\Omega}_\ell \backslash \Omega_{\ell+1}}\right\}^{\frac{1}{2}}\right\|_{L^q(\mathbb R^{n+m})}\\
\leq &C\left\|2^{\ell} \chi_{R \cap \tilde{\Omega}_\ell \backslash \Omega_{\ell+1}}\right\|_{L^q(\mathbb R^{n+m})}
\leq C2^\ell|\tilde\Omega_\ell|^{\frac{1}{q}}
\end{align*}
where the first inequality comes from the fact that 
$\mathcal{M}_s\left(\chi_{R \cap \tilde{\Omega}_\ell \backslash \Omega_{\ell+1}}\right)\left(x, y\right)>\frac{1}{2}$ and then 
$\chi_R\left(x, y\right) \leq C\mathcal{M}_s^2\left(\chi_{R \cap \tilde{\Omega}_\ell \backslash \Omega_{\ell+1}}\right)\left(x, y\right)$.
Therefore,  we get that $$\|b\|_{L^q(\mathbb R^{n+m})}\le C\frac{|\tilde \Omega_\ell|^{\frac{1}{q}}}{\|\chi_{\tilde\Omega_\ell}\|_{X(\mathbb R^{n+m})}}.$$
Moreover, we further construct the rectangle atoms. We rewrite
$b_\ell=\sum_{Q\in m(\tilde\Omega_\ell)}$, where
$$
b^Q_\ell(x,y)=\frac{1}{\lambda_\ell}\sum_{R \in B_\ell,\; R\subset Q}|R| \tilde{\phi}_R\left(x-x_I, y-y_J\right) \phi_R *\left(T_N^{-1}(f)\right)\left(x_I, y_J\right).
$$
To see the support condition and moment condition of $b^Q_\ell$,
it follows from that of $\tilde\phi$. Meanwhile, by using the above estimate for $\|b\|_{L^q({\mathbb R^{n+m}})}$, we conclude that
\begin{align*}
\sum_{Q\in m(\tilde \Omega_\ell)}\|b_\ell^Q\|^q_{L^q({\mathbb R^{n+m})}}&
=\sum_{Q\in m(\tilde \Omega_\ell)}\frac{1}{\lambda^q_\ell}
\left\|\sum_{R \in B_\ell,\; R\subset Q}|R| \tilde{\phi}_R\left(x-x_I, y-y_J\right) \phi_R *\left(T_N^{-1}(f)\right)\left(x_I, y_J\right)\right\|^q_{L^q({\mathbb R^{n+m})}}\\
&=\frac{1}{2^{\ell q}\|\chi_{\tilde\Omega}\|^q_{X(\mathbb R^{n+m})}}
\left\|\sum_{R \in B_\ell}|R| \tilde{\phi}_R\left(x-x_I, y-y_J\right) \phi_R *\left(T_N^{-1}(f)\right)\left(x_I, y_J\right)\right\|^q_{L^q({\mathbb R^{n+m})}}\\&\le C\frac{|\tilde\Omega_\ell|}{\|\chi_{\tilde\Omega_\ell}\|^q_{X(\mathbb R^{n+m})}}.
\end{align*}

Since that $$\sum_{\ell}\left(2^\ell \chi_{\Omega_\ell}\right)^s\sim
\left(\sum_{\ell}2^\ell \chi_{\Omega_\ell\setminus\Omega_{\ell+1}}\right)^s,$$ we have 
\begin{align*}
&\left\|\left\{\sum_{\ell}\left(\frac{\lambda_\ell}{\|\chi_{\Omega_\ell\|_{X(\mathbb R^{(n+m)}}}}\right)^s\chi_{\Omega_\ell}\right\}^{\frac{1}{s}}\right\|_{X(\mathbb R^{(n+m)}}
\le C\left\|\left\{\sum_{\ell} 2^{\ell s}\chi_{\Omega_\ell}\right\}^{\frac{1}{s}}\right\|_{X(\mathbb R^{(n+m)}}\\
&\le C\left\|\sum_{\ell} 2^{\ell}\chi_{\Omega_\ell\setminus\Omega_{\ell+1}}\right\|_{X(\mathbb R^{(n+m)}}
\le C\left\|\sum_{\ell}\tilde g(f)\chi_{\Omega_\ell\setminus\Omega_{\ell+1}}\right\|_{X(\mathbb R^{(n+m)}}
\\
&\sim C\left\|\tilde g(f)\right\|_{X(\mathbb R^{(n+m)}}
= C\|f\|_{H_X(\mathbb R^n\times\mathbb R^m)}.
\end{align*}
This finishes the the proof of Theorem \ref{Decom}.
\end{proof}

\begin{remark}
When $X\left(\mathbb{R}^{n}\times\mathbb R^m\right)$ is the
classical Lebesgue space $L^{p}\left(\mathbb{R}^{n}\times\mathbb R^m\right)$
with $0<s=p<1$, 
Theorem \ref{Decom} was first proved by Han et al. \cite{HLZ}.
When $X\left(\mathbb{R}^{n}\times\mathbb R^m\right)$ is
the weighted Lebesgue space associated with product Muckenhoupt weights $L_w^{p}\left(\mathbb{R}^{n}\times\mathbb R^m\right)$ with $0<s=p<1$, 
Theorem \ref{Decom} was first established by Wu \cite{Wu}.
 To our best knowledge, when $X\left(\mathbb{R}^{n}\times\mathbb R^m\right)$ is the product Herz space $\vec{K}_q^{\alpha, p}\left(\mathbb{R}^n \times \mathbb{R}^m\right)$, the weighted product Morrey space $\mathcal{M}_{u,p}^{w}\left(\mathbb{R}^{n}\times\mathbb R^m\right)$ and the product Musielak--Orlicz--type space $L^\varphi(\mathbb R^m\times\mathbb R^m)$, respectively, Theorem \ref{Decom} is totally new.
\end{remark}

\section{Extrapolation on ball quasi Banach function spaces}

We now recall the product weight functions for the product domain from $\cite{gr85}$.
For $1<p<\infty$, we say that $w\in A_{p,R}$, if there exists a constant $C>0$, such that for
any rectangles $R$ in $\mathbb{R}^n\times \mathbb R^m$,
$$\left(\frac{1}{|R|}\int_{R}w(z) dz\right)
\left(\frac{1}{|R|}\int_{R} w(z)^{-\frac{1}{p-1}}dz\right)^{p-1}<\infty.$$
We say that a nonnegative measurable function $w\in A_{1,R}$ if for
any rectangles $R$ in $\mathbb{R}^n\times \mathbb R^m$
$$
\left(\frac{1}{|R|}\int_{R}w(z) dz\right)
\mbox{ess}\;\sup_{z\in R}w(z)^{-1}<\infty.
$$
A weight $w$ belongs to $A_{\infty,R}$, if there is $\delta>0$ and $C>0$ such that
for any rectangles $R$ in $\mathbb{R}^n\times \mathbb R^m$ and all measurable subsets $E\subset R$
$$
\frac{w(A)}{w(R)}\le C\left(\frac{|A|}{|R|}\right)^\delta.
$$
The relation between $A_{\infty,R}$ and $A_{p,R}$ is
$$A_{\infty,R}=\cup_{1\le p<\infty}A_{p,R}.$$

The known boundeness of the strong maximal operators
$\mathcal M_S$ is stated as follows.

\begin{proposition}\cite{gr85}\label{max}\quad Let $1<p<\infty$ and $w\in A_{p,R}.$ Then $\mathcal M_s$ is of type
$(L^p_w(\mathbb R^n\times\mathbb R^m),L^p_w(\mathbb R^n\times\mathbb R^m))$.
\end{proposition}

First we need the $A_\infty$ extrapolation in weighted product domain $\mathbb R^n\times\mathbb R^m$. The $A_\infty$ extrapolation for the general Muckenhoupt basis in $\mathbb R^n$ was proved in \cite[Corollary 3.15]{CMP}. Now we state the following $A_\infty$ extrapolation in weighted product domain $\mathbb R^n\times\mathbb R^m$,
whose proof is nearly identical to the corresponding one of \cite[Corollary 3.15]{CMP}]; we omit the details.

\begin{theorem}\label{s2th1}\quad Let $\mathcal{F}$ denote a family of ordered pairs of
non-negative measurable functions $(f, g)$. Assume that for some $p_0$ with $0<p_0<\infty$
and every weight $w_0\in A_{\infty,R}$,
\begin{align}\label{s2i1}
\int_{\mathbb{R}^n\times \mathbb R^m} f(x,y)^{p_0}w_0(x,y)dxdy\leq C_0
\int_{\mathbb{R}^n\times \mathbb R^m} g(x,y)^{p_0}w_0(x,y)dxdy,\quad (f, g)\in \mathcal{F}.
\end{align}

Then for all $p$ such that $0<p<\infty$ and for all
$w\in {A_{\infty,R}}$,
\begin{align}\label{s2i2}
\int_{\mathbb{R}^n\times \mathbb R^m} f(x,y)^{p}w(x,y)dxdy\leq C_0
\int_{\mathbb{R}^n\times \mathbb R^m} g(x,y)^{p}w(x,y)dxdy,\quad (f, g)\in \mathcal{F}.
\end{align}

Moreover, for all $0<p, r<\infty$, for all
$w\in {A_{\infty,R}}$, and all sequences $\{(f_j,g_j)\}_j\subset \mathcal F$,
\begin{align}\label{s2i3}
  \left\|\left(\sum_j|f_j|^r\right)^{1/r}\right\|_{L_w^{p}(\mathbb{R}^n\times \mathbb R^m)}\leq
  C\left\|\left(\sum_j|g_j|^r\right)^{1/r}\right\|_{L_w^{p}(\mathbb{R}^n\times \mathbb R^m)}.
\end{align}
\end{theorem}

The boundedness of the bi-parameter singular integral operator $\tilde T$ on the weighted product Hardy spaces are given in \cite[Theorem 1.2]{DHLW}.

\begin{theorem}\label{weighted}
Suppose that $w\in A_{\infty, R}$.
Then the bi-parameter singular integral operator $\tilde T$ is bounded from weighted product Hardy space ${H}_w^p(\mathbb R^n\times\mathbb R^m)$ to itself for $0<p<\infty$
and bounded from weighted product Hardy space 
${H}_w^p(\mathbb R^n\times\mathbb R^m)$ to $L_w^p(\mathbb R^{n+m})$ for $0<p\le 1$.
\end{theorem}

A direct corollary of Theorem \ref{s2th1} and Theorem \ref{weighted} is as follows.
\begin{corollary}
Suppose that $w\in {A_{\infty,R}}$ and $0<p<\infty$.
Then the bi-parameter singular integral operator $\tilde T$ is bounded from product Hardy space ${H}_w^p(\mathbb R^n\times\mathbb R^m)$ to $L_w^p(\mathbb R^n\times\mathbb R^m)$.
\end{corollary}

Similarly to \cite[Theorem 4.6]{CMP}, we get the following extrapolation theorem, 
whose proof is a slight modification of the corresponding one of \cite[Theorem 4.6]{CMP}.
\begin{theorem}\label{A1}\quad 
Let $X(\mathbb R^n\times \mathbb R^m)$ be a ball quasi-Banach function space and $X'(\mathbb R^n\times \mathbb R^m)$ be the associate space.
Given a family $\mathcal F$, assume that (\ref{s2i1}) holds for some $p_0$ with $0<p_0<\infty$
and every weight $w\in A_{1,R}$. 
If there exists $q_0$, $p_0\le q_0<\infty$, such that $X^{\frac{1}{q_0}}(\mathbb R^n\times \mathbb R^m)$ is a ball Banach function space, and the strong Hardy--Littlewood maximal operator $\mathcal M_s$ is bounded on $(X^{\frac{1}{q_0}})'(\mathbb R^n\times \mathbb R^m)$, then for any $(f, g)\in \mathcal{F}$ and
$f \in X(\mathbb R^n\times \mathbb R^m)$, we have
\begin{align*}
\|f\|_{X(\mathbb R^n\times \mathbb R^m)}\leq C\|g\|_{X(\mathbb R^n\times \mathbb R^m)}.
\end{align*}
Furthermore, for all $0<p<\infty$, for every $p$, $\frac{p_0}{q_0}\le p<\infty$, then for any $(f, g)\in \mathcal{F}$ and
$f \in X(\mathbb R^n\times \mathbb R^m)$, we have
\begin{align*}
\|f\|_{X^p(\mathbb R^n\times \mathbb R^m)}\leq C\|g\|_{X^p(\mathbb R^n\times \mathbb R^m)}.
\end{align*}
\end{theorem}

As an corollary of Theorem \ref{s2th1} and Theorem \ref{A1}, we can extend the $A_\infty$ extrapolation theorem to the ball quasi Banach function spaces on product domains.

\begin{corollary}\label{Ainfty}\quad 
Let $X(\mathbb R^n\times \mathbb R^m)$ be a ball quasi-Banach function space and $X'(\mathbb R^n\times \mathbb R^m)$ be the associate space.
Given a family $\mathcal F$, assume that (\ref{s2i1}) holds for some $p_0$ with $0<p_0<\infty$
and every weight $w\in A_{\infty,R}$. 
If there exists $q_1$, $0<q_1<\infty$, such that $X^{\frac{1}{q_1}}(\mathbb R^n\times \mathbb R^m)$ is ball Banach function space, and the strong Hardy--Littlewood maximal operator $\mathcal M_s$ is bounded on $(X^{\frac{1}{q_1}})'(\mathbb R^n\times \mathbb R^m)$, then  for any $(f, g)\in \mathcal{F}$ and
$f \in X(\mathbb R^n\times \mathbb R^m)$, we have
\begin{align}\label{s2i4}
\|f\|_{X(\mathbb R^n\times \mathbb R^m)}\leq C\|g\|_{X(\mathbb R^n\times \mathbb R^m)}.
\end{align}
Furthermore, for all $0<q<\infty$, for all sequences $\{(f_j,g_j)\}_j\subset \mathcal F$,
\begin{align}\label{s2i3}
  \left\|\left(\sum_j|f_j|^q\right)^{1/q}\right\|_{X(\mathbb{R}^n\times \mathbb R^m)}\leq
  C\left\|\left(\sum_j|g_j|^q\right)^{1/q}\right\|_{X(\mathbb{R}^n\times \mathbb R^m)}.
\end{align}
\end{corollary}

Applying the similar argument to the proof of \cite[Theorem 10.1]{Cr}, we can get the desired $A_p$ extrapolation. Hence, we omit the details.
\begin{theorem}\label{Ap}\quad 
Let $X(\mathbb R^n\times \mathbb R^m)$ be a ball quasi-Banach function space, and $X'(\mathbb R^n\times \mathbb R^m)$ be the associate space of $X(\mathbb R^n\times \mathbb R^m)$. Suppose that the strong Hardy--Littlewood maximal operator $\mathcal M_s$ is bounded on $X(\mathbb R^n\times \mathbb R^m)$ and $X'(\mathbb R^n\times \mathbb R^m)$.
Given a family $\mathcal F$, assume that (\ref{s2i1}) holds for some $p_0$ with $1\le p_0<\infty$
and every weight $w\in A_{p_0,R}$. Then for any $(f, g)\in \mathcal{F}$ and
$f \in X(\mathbb R^n\times \mathbb R^m)$, we have
\begin{align}\label{s2i4}
\|f\|_{X(\mathbb R^n\times \mathbb R^m)}\leq C\|g\|_{X(\mathbb R^n\times \mathbb R^m)}.
\end{align}
Futhermore, for all $1<q<\infty$, for all sequences $\{(f_j,g_j)\}_j\subset \mathcal F$,
\begin{align}\label{s2i3}
  \left\|\left(\sum_j|f_j|^q\right)^{1/q}\right\|_{X(\mathbb{R}^n\times \mathbb R^m)}\leq
  C\left\|\left(\sum_j|g_j|^q\right)^{1/q}\right\|_{X(\mathbb{R}^n\times \mathbb R^m)}.
\end{align}
\end{theorem}

To apply Theorem \ref{Ap} and Proposition \ref{max}, we can
get the following Fefferman--Stein vector-valued maximal inequality
on ball Banach function space.

\begin{corollary}\label{Fs1}\quad 
Let $X(\mathbb R^n\times \mathbb R^m)$ be a ball quasi-Banach function space, and $X'(\mathbb R^n\times \mathbb R^m)$ be the associate space of $X(\mathbb R^n\times \mathbb R^m)$. Suppose that the strong Hardy--Littlewood maximal operator $\mathcal M_s$ is bounded on $X(\mathbb R^n\times \mathbb R^m)$ and $X'(\mathbb R^n\times \mathbb R^m)$.
Then for all $1<q<\infty$, 
\begin{align*}
  \left\|\left(\sum_j|\mathcal M_s(f_j)|^q\right)^{1/q}\right\|_{X(\mathbb{R}^n\times \mathbb R^m)}\leq
  C\left\|\left(\sum_j|f_j|^q\right)^{1/q}\right\|_{X(\mathbb{R}^n\times \mathbb R^m)}.
\end{align*}
\end{corollary}

The weighted norm inequality
for Littlewood-Paley square function $\mathcal G(f)$ on product Lebesgue spaces is obtained in \cite{DHLW}.
Let $1<p<\infty$, $w\in A_{p,R}$. Then there exist constants $C_1$ and
$C_2$ depending on $p$ such that
\begin{align*}
C\|f\|_{L_w^{p}}\le \|\mathcal G_d(f)\|_{L_w^{p}}\le C\|f\|_{L_w^{p}}.
\end{align*}

Thus, by applying Theorem \ref{Ap}, we conclude the following result.

\begin{theorem}
Let $X(\mathbb R^n\times \mathbb R^m)$ be a ball quasi-Banach function space, and $X(\mathbb R^n\times \mathbb R^m)$ be the associate space of $X(\mathbb R^n\times \mathbb R^m)$. Suppose that the strong Hardy--Littlewood maximal operator $\mathcal M_s$ is bounded on $X(\mathbb R^n\times \mathbb R^m)$ and $X'(\mathbb R^n\times \mathbb R^m)$. Then 
$$
{H}_X(\mathbb R^n\times\mathbb R^m)\sim X(\mathbb R^n\times\mathbb R^m)
$$
with equivalent norms.
\end{theorem}

We state the following result on weighted product Hardy spaces and weighted Lebesgue spaces in \cite[Theorem 3.5]{DHLW}.
\begin{theorem}\label{Hardy}
Suppose that $w\in A_{\infty, R}$.
If $f\in{H}_w^p(\mathbb R^n\times\mathbb R^m)\cap L^2(\mathbb R^{n+m})$
with $0<p\le 1$, then $f\in L_w^p(\mathbb R^{n+m})$ with
$$
\|f\|_{L_w^p(\mathbb R^{n+m})}\le C\|f\|_{H_w^p(\mathbb R^n\times\mathbb R^m)}.
$$
\end{theorem}

Therefore, Theorems \ref{Ainfty} and Theorem \ref{Hardy} yield the following result.
\begin{theorem}\label{s4t2}
Let $X(\mathbb R^n\times \mathbb R^m)$ be a ball quasi-Banach function space satisfying Assumption \ref{as:01} with some
$0<\theta<s\le1$ and $X'(\mathbb R^n\times \mathbb R^m)$ be the associate space.
If there exists $q_1$, $0<q_1<\infty$, such that $X^{\frac{1}{q_1}}(\mathbb R^n\times \mathbb R^m)$ is ball quasi-Banach function space, and the strong Hardy--Littlewood maximal operator $\mathcal M_s$ is bounded on $(X^{\frac{1}{q_1}})'(\mathbb R^n\times \mathbb R^m)$.
If $f\in{H}_X(\mathbb R^n\times\mathbb R^m)\cap L^2(\mathbb R^{n+m})$, then $f\in X(\mathbb R^n\times\mathbb R^m)$ with
$$
\|f\|_{X(\mathbb R^n\times\mathbb R^m)}\le C\|f\|_{H_X(\mathbb R^n\times\mathbb R^m)}.
$$
\end{theorem}

Finally, we explore a general approach to derive $H_X\rightarrow X$ boundedness
from $H_X\rightarrow H_X$ boundedness of linear operators. 

\begin{theorem}\label{s4t3} 
Let $X(\mathbb R^n\times \mathbb R^m)$ be a ball quasi-Banach function space satisfying Assumption \ref{as:01} with some
$0<\theta<s\le1$ and $X'(\mathbb R^n\times \mathbb R^m)$ be the associate space.
If there exists $q_1$, $0<q_1<\infty$, such that $X^{\frac{1}{q_1}}(\mathbb R^n\times \mathbb R^m)$ is ball quasi-Banach function space, and the strong Hardy--Littlewood maximal operator $\mathcal M_s$ is bounded on $(X^{\frac{1}{q_1}})'(\mathbb R^n\times \mathbb R^m)$.
Assume further that $X(\mathbb R^n\times\mathbb R^m)$ has an absolutely
continuous quasi-norm and $2\leq q<\infty$.
Then any linear operator $T$ which is bounded both on $L^2\left(\mathbb{R}^{n+m}\right)$ and $H_X\left(\mathbb{R}^n \times \mathbb{R}^m\right)$, is bounded from $H_X\left(\mathbb{R}^n \times\mathbb{R}^m\right)$ to $X\left(\mathbb{R}^{n+m}\right)$.
\end{theorem}
\begin{proof}
For $f \in L^2\left(\mathbb{R}^{n+m}\right) \cap H_X\left(\mathbb{R}^n \times \mathbb{R}^m\right)$, by applying Theorem \ref{s4t2} we conclude that
$$
\|T(f)\|_{X\left(\mathbb{R}^{n+m}\right)} \leq C\|T(f)\|_{H_X\left(\mathbb{R}^n \times \mathbb{R}^m\right)} \leq C\|f\|_{H_X\left(\mathbb{R}^n \times \mathbb{R}^m\right)} .
$$
Therefore, by using a density argument, we have completed the proof of Theorem \ref{s4t3}.
\end{proof}

\section{Boundedness of bi-parameter singular integrals}
In this section, we aim to obtain the ${H}_X(\mathbb R^n\times\mathbb R^m)$ boundedness of the bi-parameter singular integral operator $T$ via extrapolation. First we will explore a general method to derive 
${H}_X(\mathbb R^n\times\mathbb R^m)\rightarrow X(\mathbb R^n\times\mathbb R^m)$ boundedness of the general operators $T$ via extrapolation. As an application, we conclude that the bi-parameter singular integral operator $T$ are bounded from ${H}_X(\mathbb R^n\times\mathbb R^m)$ to itself and from ${H}_X(\mathbb R^n\times\mathbb R^m)$ to $X(\mathbb R^n\times\mathbb R^m)$.

In order to obtain the mapping properties on some concrete function spaces, we usually need a dense subset of these spaces.
 By applying the following refined result, we will find that the density argument is not required. 
Therefore, we can get rid of the the additional assumption of absolute continuity of the quasi-norm for $X$, which means that our method can be applied to more general function spaces. Moreover, we do not require that the operator $T$ is any linear or sublinear.
Therefore, the following result also applies to general nonlinear operators.
We remark that some extensions of the theory of extrapolations and applications on $\mathbb R^n$ can be found in \cite[Theorem 3.10]{Ho2} , \cite[Theorem 3]{Ho22}, \cite[Theorem 5.8]{Tan23} and  \cite[Proposition 2.14]{TYYZ}.

\begin{theorem}\label{Rsh}\quad 
Let $q_0\in (0,\infty)$, $X(\mathbb R^n\times \mathbb R^m)$ be a ball quasi-Banach function space and $X'(\mathbb R^n\times \mathbb R^m)$ be the associate space.
Given a family $\mathcal F$, assume that the operator $T$ fulfills that
\begin{align*}
\int_{\mathbb{R}^n\times \mathbb R^m} Tf(x,y)^{q_0}w_0(x,y)dxdy\leq C_0
\int_{\mathbb{R}^n\times \mathbb R^m} f(x,y)^{q_0}w_0(x,y)dxdy,\quad (Tf, f)\in \mathcal{F}
\end{align*} 
holds for every weight 
$$w_0\in \left\{\mathcal R_sh:h\in (X^{{\frac{1}{q_0}}})'(\mathbb R^n\times \mathbb R^m)\right\}.$$
If $X^{\frac{1}{q_0}}(\mathbb R^n\times \mathbb R^m)$ is a ball Banach function space, and the strong Hardy--Littlewood maximal operator $\mathcal M_s$ is bounded on $(X^{\frac{1}{q_0}})'(\mathbb R^n\times \mathbb R^m)$, then for any $(Tf, f)\in \mathcal{F}$ and
$f \in X(\mathbb R^n\times \mathbb R^m)$, we have
\begin{align*}
\|Tf\|_{X(\mathbb R^n\times \mathbb R^m)}\leq C\|f\|_{X(\mathbb R^n\times \mathbb R^m)}.
\end{align*}
\end{theorem}

\begin{proof}
Let $Y(\mathbb R^n\times \mathbb R^m)=X^{\frac{1}{q_0}}(\mathbb R^n\times \mathbb R^m)$. Since  $\mathcal M_s$ is bounded on $Y'(\mathbb R^n\times \mathbb R^m)$, then we can define the Rubio de Francia iteration algorithm for any $h\in L_{loc}(\mathbb R^n\times \mathbb R^m)$:
$$
\mathcal R_sh(x)=\sum_{k=0}^\infty\frac{\mathcal M^kh(x)}{2^k\|\mathcal M\|^k_{(Y'(\mathbb R^n\times \mathbb R^m)\rightarrow Y'(\mathbb R^n\times \mathbb R^m))}},
$$
where $\mathcal M^k$ is the iterations of the operator $\mathcal M$.
From the definition of $\mathcal R_s$, for any $h\in Y'(\mathbb R^n\times \mathbb R^m)$ we get that
\begin{align}\label{R1}
|h(x)|\le \mathcal R_sh(x), 
\end{align}
\begin{align}\label{R2}
\|\mathcal R_sh\|_{Y'(\mathbb R^n\times \mathbb R^m)}\le 2\|h\|_{Y'(\mathbb R^n\times \mathbb R^m)}, 
\end{align}
\begin{align}\label{R3}
[\mathcal R_s h]_{A_{1,R}}\le 2\|\mathcal M_s\|_{(Y'(\mathbb R^n\times \mathbb R^m)\rightarrow Y'(\mathbb R^n\times \mathbb R^m))}. 
\end{align}
In fact, $$
\mathcal M_s(\mathcal R_sh)\le \sum_{k=0}^\infty\frac{\mathcal M^{k+1}h}{2^k\|\mathcal M_s^k\|_{(Y'(\mathbb R^n\times \mathbb R^m)\rightarrow Y'(\mathbb R^n\times \mathbb R^m))}}
\le 2\|\mathcal M_s\|_{(Y'(\mathbb R^n\times \mathbb R^m)\rightarrow Y'(\mathbb R^n\times \mathbb R^m))}\mathcal R_sh
$$
and then $\mathcal R_sh\in A_{1,R}$.
Let $f\in X(\mathbb R^n\times\mathbb R^m)$, for any $h\in Y'(\mathbb R^n\times \mathbb R^m)$, we conclude that
\begin{align*}
&\int_{\mathbb R^n\times \mathbb R^m}|f(x,y)|^{q_0}\mathcal R_sh(x,y)dxdy
\le \| |f|^{q_0} \|_{Y(\mathbb R^n\times \mathbb R^m)}\|\mathcal R_sh\|_{Y'(\mathbb R^n\times \mathbb R^m)}\le\|f\|^{q_0}_{X(\mathbb R^n\times \mathbb R^m)}\|h\|_{Y'(\mathbb R^n\times \mathbb R^m)},
\end{align*}
where the last inequality deduces from (\ref{R2}).
Therefore, we have
$$
X(\mathbb R^n\times \mathbb R^m)\hookrightarrow \bigcap_{h\in Y'(\mathbb R^n\times \mathbb R^m)} L^{q_0}_{\mathcal R_sh}(\mathbb R^n\times \mathbb R^m).
$$

Notice that $\mathcal R_sh\in A_{1,R}\subset A_{A_{q_0,R}}$.
For any $h\in Y'(\mathbb R^n\times \mathbb R^m)$ and $(\ref{R1})$, 
by using the fact that $T$ is bounded on $L^{q_0}_{\mathcal R_sh}(\mathbb R^n\times \mathbb R^m)$, we obtain that
\begin{align*}
&\int_{\mathbb R^n\times \mathbb R^m}|Tf(x,y)|^{q_0}h(x,y)dxdy\\
&\le \int_{\mathbb R^n\times \mathbb R^m}|Tf(x,y)|^{q_0}\mathcal R_sh(x,y)dxdy\\
&\le \int_{\mathbb R^n\times \mathbb R^m}|f(x,y)|^{q_0}\mathcal R_sh(x,y)dxdy.
\end{align*}

Then by applying the H\"older inequality and (\ref{R2}), we deduce that
\begin{align*}
&\int_{\mathbb R^n\times \mathbb R^m}|Tf(x,y)|^{q_0}h(x,y)dxdy\le C\| |f|^{q_0} \|_{Y(\mathbb R^n\times \mathbb R^m)}\| \mathcal R_s h\|_{Y'(\mathbb R^n\times \mathbb R^m)}\le C\|f\|^{q_0}_{X(\mathbb R^n\times \mathbb R^m)}\|h\|_{Y'(\mathbb R^n\times \mathbb R^m)}
\end{align*}
for any $h\in Y'(\mathbb R^n\times \mathbb R^m)$.
We now prove the desired inequality.
 For any $(Tf, f)\in \mathcal{F}$ and
$f \in X(\mathbb R^n\times \mathbb R^m)$, since $Y$ is a ball Banach function space we have
\begin{align*}
&\|Tf\|^{q_0}_{X(\mathbb R^n\times \mathbb R^m)}
=\| |Tf|^{q_0}\|_{Y(\mathbb R^n\times \mathbb R^m)}\\
&=\sup \left\{\int_{\mathbb R^n\times \mathbb R^m} |Tf(x,y)|^{q_0}h(x,y)dxdy:\; h\in Y'(\mathbb R^n\times\mathbb R^m),\,\|h\|_{Y'(\mathbb R^n\times \mathbb R^m)}\le 1\right\}\\
&\le C\|f\|^{q_0}_{X(\mathbb R^n\times \mathbb R^m)}.
\end{align*}
Since the constant $C$ is independent of $h$, the proof is completed.
\end{proof}

Next we obtain the mapping properties for the product Hardy spaces ${H}_X(\mathbb R^n\times\mathbb R^m)$. A general method for get the boundedness of operators on the product Hardy spaces ${H}_X(\mathbb R^n\times\mathbb R^m)$ is as follows by applying Theorem \ref{Rsh}.

\begin{theorem}\label{general}\quad 
Let $q_0\in (0,\infty)$, $X(\mathbb R^n\times \mathbb R^m)$ be a ball quasi-Banach function space and $X'(\mathbb R^n\times \mathbb R^m)$ be the associate space.
Given a family $\mathcal F$, suppose that 
for every weight 
$$w_0\in \left\{\mathcal R_sh:h\in (X^{{\frac{1}{q_0}}})'(\mathbb R^n\times \mathbb R^m)\right\}$$
the operator $T:H^{q_0}_w\rightarrow L^{q_0}_w$ is bounded.
If $X^{\frac{1}{q_0}}(\mathbb R^n\times \mathbb R^m)$ is a ball Banach function space, and the strong Hardy--Littlewood maximal operator $\mathcal M_s$ is bounded on $(X^{\frac{1}{q_0}})'(\mathbb R^n\times \mathbb R^m)$, then for any $(Tf, f)\in \mathcal{F}$ and
$f \in H_X(\mathbb R^n\times \mathbb R^m)$, we have
\begin{align*}
\|Tf\|_{X(\mathbb R^n\times \mathbb R^m)}\leq C\|f\|_{H_X(\mathbb R^n\times \mathbb R^m)}.
\end{align*}
\end{theorem}

\begin{proof}
Let $f\in H_X(\mathbb R^n\times \mathbb R^m)$.
Then we know that 
$$\mathcal G_d(f)\in X(\mathbb R^n\times \mathbb R^m)\hookrightarrow \bigcap_{h\in (X^{{1/q_0}})'(\mathbb R^n\times \mathbb R^m),\, \|h\|_{(X^{1/{q_0}})'(\mathbb R^n\times \mathbb R^m)}\le 1} L^{q_0}_{\mathcal R_sh}(\mathbb R^n\times \mathbb R^m),$$
which implies that
$$
f\in \bigcap_{h\in (X^{{1/q_0}})'(\mathbb R^n\times \mathbb R^m),\, \|h\|_{(X^{1/{q_0}})'(\mathbb R^n\times \mathbb R^m)}\le 1} H^{q_0}_{\mathcal R_sh}(\mathbb R^n\times \mathbb R^m).
$$
Thus, 
$$
H_X(\mathbb R^n\times \mathbb R^m)\hookrightarrow \bigcap_{h\in (X^{{1/q_0}})'(\mathbb R^n\times \mathbb R^m),\, \|h\|_{(X^{1/{q_0}})'(\mathbb R^n\times \mathbb R^m)}\le 1} H^{q_0}_{\mathcal R_sh}(\mathbb R^n\times \mathbb R^m).
$$
For any $f\in H_X(\mathbb R^n\times \mathbb R^m)$ and 
$$w_0\in \left\{\mathcal R_sh:h\in (X^{{\frac{1}{q_0}}})'(\mathbb R^n\times \mathbb R^m),\;\|h\|_{(X^{1/{q_0}})'}\le 1\right\},$$
from the fact the boundedness of $T$, we conclude that
\begin{align*}
\int_{\mathbb{R}^n\times \mathbb R^m} Tf(x,y)^{q_0}w_0(x,y)dxdy\leq C_0
\int_{\mathbb{R}^n\times \mathbb R^m} \mathcal G_df(x,y)^{q_0}w_0(x,y)dxdy,\quad (Tf, f)\in \mathcal{F}.
\end{align*} 
By applying Theorem \ref{Rsh} with
$$\mathcal F=\left\{(|Tf|,\;\mathcal G_df): f\in H_X(\mathbb{R}^n\times \mathbb R^m)\right\},$$
we have
$$
\|Tf\|_{X(\mathbb{R}^n\times \mathbb R^m)}\le C\|\mathcal G_df\|_{X(\mathbb{R}^n\times \mathbb R^m)}=C\|f\|_{H_X(\mathbb{R}^n\times \mathbb R^m)}.
$$
Therefore, we have completed the proof of Theorem \ref{general}.\end{proof}

We now present the boundedness of  the bi-parameter singular integral operator $\widetilde T$ on product Hardy space ${H}_X(\mathbb R^n\times\mathbb R^m)$.
\begin{theorem}\label{SIO}
Let $q_1\in (0,\infty)$, $X(\mathbb R^n\times \mathbb R^m)$ be a ball quasi-Banach function space satisfying Assumption \ref{as:01} with some
$0<\theta<s\le1$ and $X'(\mathbb R^n\times \mathbb R^m)$ be the associate space.
If $X^{\frac{1}{q_1}}(\mathbb R^n\times \mathbb R^m)$ is ball quasi-Banach function space, and the strong Hardy--Littlewood maximal operator $\mathcal M_s$ is bounded on $(X^{\frac{1}{q_1}})'(\mathbb R^n\times \mathbb R^m)$.
Then the bi-parameter singular integral operator $\widetilde T$ is bounded from product Hardy space ${H}_X(\mathbb R^n\times\mathbb R^m)$ to itself
and bounded from product Hardy space ${H}_X(\mathbb R^n\times\mathbb R^m)$ to $X(\mathbb R^n\times\mathbb R^m)$.
\end{theorem}
\begin{proof}
From Theorem \ref{weighted}, we know that for any 
$w\in A_{\infty, R}$, $\widetilde T$ is bounded from weighted product Hardy space ${H}_w^p(\mathbb R^n\times\mathbb R^m)$ to itself for $0<p<\infty$.
Notice that
$$\left\{\mathcal R_sh:h\in (X^{{\frac{1}{q_1}}})'(\mathbb R^n\times \mathbb R^m),\;\|h\|_{(X^{1/{q_1}})'}\le 1\right\}\subset A_{1,R}\subset A_{\infty, R}.$$
By applying Theorem \ref{general},
we get that $\widetilde T$ is bounded from product Hardy space ${H}_X(\mathbb R^n\times\mathbb R^m)$ to $X(\mathbb R^n\times\mathbb R^m)$.
On the other hand, denote that $Uf:=\mathcal G_d \widetilde Tf$,
Theorem \ref{weighted} guarantees that for any $w\in A_{\infty, R}$,
$U$ is bounded from weighted product Hardy space 
${H}_w^{q_1}(\mathbb R^n\times\mathbb R^m)$ to $L_w^{q_1}(\mathbb R^{n+m})$ for any $0<q_1\le 1$.
Therefore, Theorem \ref{general} assert that
$U: H_X(\mathbb R^n\times\mathbb R^m)\rightarrow X(\mathbb R^n\times\mathbb R^m)$ is bounded.
Hence, 
$$\|\widetilde Tf\|_{H_X(\mathbb R^n\times\mathbb R^m)}=\|Uf\|_{X(\mathbb R^n\times\mathbb R^m)}\le C\|f\|_{H_X(\mathbb R^n\times\mathbb R^m)}.$$
This ends the proof of Theorem \ref{SIO}.
\end{proof}

\begin{remark}
We remark that the boundedness of the bi-parameter singular integral operator $\widetilde T$ on the classical product Hardy space $H^{p}\left(\mathbb{R}^{n}\times\mathbb R^m\right)$
was obtained in \cite{HLZ}
and that the corresponding result for the weighted product Hardy space associated with product Muckenhoupt weights $H_w^{p}\left(\mathbb{R}^{n}\times\mathbb R^m\right)$
was established in \cite{DHLW}.	
 To our best knowledge, our results on the boundedness of the bi-parameter singular integral operator $\widetilde T$ for the product Herz--Hardy space  $H\vec{K}_q^{\alpha, p}\left(\mathbb{R}^n \times \mathbb{R}^m\right)$, the weighted product Hardy--Morrey space $H\mathcal{M}_{u,p}^{w}\left(\mathbb{R}^{n}\times\mathbb R^m\right)$ and the product Musielak--Orlicz--Hardy--type space $H^\varphi(\mathbb R^n\times\mathbb R^m)$
are completely new. 
\end{remark}

\section*{Acknowledgments}This project is supported by the National Natural Science Foundation of China (Grant No. 11901309), Natural Science Foundation of Jiangsu Province of China (Grant No. BK20180734) and Natural Science Foundation of Nanjing University of Posts and Telecommunications (Grant No. NY222168).


\bigskip
\noindent Jian Tan\\
\noindent School of Science\\ Nanjing University of Posts and Telecommunications\\
    Nanjing 210023, People's Republic of China

\noindent {\it E-mail address}: \texttt{tanjian89@126.com; tj@njupt.edu.cn}

\end{document}